\documentclass[10pt]{article}
\usepackage{color}
\usepackage{amssymb}

\usepackage{amsthm,array,amssymb,amscd,amsfonts,latexsym, url}
\usepackage{amsmath}

\newtheorem{theo}{Theorem}[section]
\newtheorem{prop}[theo]{Proposition}
\newtheorem{lemm}[theo]{Lemma}
\newtheorem{coro}[theo]{Corollary}
\newtheorem{rema}[theo]{Remark}
\newtheorem{Defi}[theo]{Definition}

\newtheorem{claim}[theo]{Claim}

\newtheorem{conj}[theo]{Conjecture}

\title{Some new   results on  modified diagonals}
\author{Claire Voisin\footnote{This research has been   supported by  The Charles Simonyi Fund and The Fernholz Foundation.}
\\CNRS, Institut de Math\'ematiques de Jussieu and IAS}
\date{}

\newfont{\gothic}{eufb10}
\begin{document}
\maketitle

\begin{abstract} In the paper \cite{ogrady}, O'Grady studied
$m$-th modified diagonals for a smooth connected projective variety,
generalizing the Gross-Schoen modified small diagonal
\cite{grossschoen}. These cycles $\Gamma^m(X,a)$ depend on a choice
of reference point $a\in X$ (or more generally a degree $1$
zero-cycle). We prove that for any $X,a$, the cycle $\Gamma^m(X,a)$
vanishes for large $m$. We also prove  the following conjecture  of
O'Grady: if $X$ is a double  cover of $Y$ and $\Gamma^m(Y,a)$
vanishes (where $a$ belongs to the branch locus), then
$\Gamma^{2m-1}(X,a)$ vanishes, and we provide a   generalization to
higher degree finite covers.
 We finally prove the vanishing
$\Gamma^{n+1}(X,o_X)=0$ when $X=S^{[m]}$, $S$ a $K3$ surface, and $n=2m$, which was
conjectured by O'Grady and proved by him for $m=2,3$.
 \end{abstract}

\section{Introduction}
Let $X$ be a connected smooth projective  variety of dimension $n$.
We will denote in this paper ${\rm CH}_i(X)$ the Chow groups of $X$
with rational coefficients and ${\rm CH}_i(X)/{\rm alg}$  the groups
of $i$-cycles of $X$ with $\mathbb{Q}$-coefficients modulo algebraic
equivalence.
 Let
$a\in{\rm CH}_0(X)$ be a $0$-cycle of degree $1$ on $X$. Following
Gross-Schoen \cite{grossschoen} and O'Grady \cite{ogrady}, let us
consider for $m\geq2$ the following $n$-cycle $\Gamma^m(X,a)$ in
$X^m$, which is a modification of the $m$-th small diagonal of $X$:
\begin{eqnarray} \label{eqformgammam} \Gamma^m(X,a)=\sum_{I\subset
\{1,\ldots,m\},|I|=i<m}(-1)^ip_I^*(a^{*i})\cdot p_J^*\Delta_{m-i}\in {\rm
CH}_n(X^m)_\mathbb{Q}, \end{eqnarray}
 where
 \begin{itemize}
\item $\{1,\ldots,m\}$ is the disjoint union of $I$ and $J$,
\item $p_I:X^m\rightarrow X^i$, resp. $p_J:X^m\rightarrow X^{m-i}$ are the
projections onto the products of factors indexed by $I$, resp. $J$,
\item  $\Delta_{m-i}$ is the small diagonal of $X^{m-i}$, $\Delta_1=X$,
\item $a^{*i}\in {\rm
CH}_0(X^i)$ is defined by
\begin{eqnarray} \label{eqstar} a^{*i}=p_1^*a\cdot\ldots \cdot p_i^*a.
\end{eqnarray}
\end{itemize}
For example, for $m=2$, we have $\Gamma^2(X,a)=\Delta_X-a\times
X-X\times a$ and $\Gamma^2(X,a)=0$ if and only if $X=\mathbb{P}^1$
or a point. The modified small diagonal $\Gamma^3(X,a)$ appears in
several recent works. Gross and Schoen prove that $\Gamma^3(X,a)=0$
if $X$ is a hyperelliptic curve and $a$ is a Weierstrass point. This
result was greatly extended    in \cite{colombovgeemen}
 by Colombo and
van Geemen, who  worked with
$1$-cycles modulo algebraic equivalence and proved that, for a
$d$-gonal curve $X$, the  cycle $\Gamma^{d+1}
(X,a)$
 is
algebraically equivalent to $0$. Although they do not state their
result in this form, but as the vanishing modulo algebraic
equivalence of the components $Z_s,\,s\geq d-1$ of the Beauville
decomposition (see \cite{beau}) of $X$ in its Jacobian, one can show
that this is equivalent to the vanishing of $\Gamma^{d+1}(X,a)$
modulo algebraic equivalence. For completeness, we will prove this
fact in subsection \ref{subsec}.

Concerning higher dimensional varieties,
 Beauville and the author proved in \cite{beauvoi} the following theorem:
\begin{theo}\label{BV} Let  $X$ be a $K3$ surface. Then there exists a
canonical degree $1$
zero-cycle $o_X$ of $X$ such that
\begin{eqnarray}
\Gamma^3(X,o_X)=0 \,\,{\rm in}\,\,{\rm CH}_2(X^3).
\label{eqformk3}
\end{eqnarray} In fact,
   $o_X$  can be defined as the class in ${\rm CH}_0(X)$ of any point of $X$ lying on a (singular)
   rational curve in $X$.
\end{theo}

In the paper \cite{ogrady}, O'Grady investigates $\Gamma^m(X,a)$ for
higher $m$. He proves the following results (for $X$ smooth
projective connected):
\begin{theo}\label{theoog} (O'Grady \cite{ogrady}) (i) The cycle $\Gamma^{n+1}(X,a)$ is
cohomologous to $0$, for $n={\rm dim}\,X$ and $q(X)=0$. More
generally $\Gamma^{m+1}(X,a)$ is cohomologous to $0$ if and only if
$m\geq {\rm dim}\,X+d$, where $d$ is the dimension of the image of
$X$ in its Albanese variety.

(ii) If $\Gamma^m(X,a)=0$ then $\Gamma^{m'}(X,a)=0$ for $m'\geq m$.

(iii)  If $p:X\rightarrow Y$ is a ramified double cover and $a$ is a
branch point such that  $\Gamma^m(Y, a)=0$, then  for $m=2$ or
$m=3$, $\Gamma^{2m-1}(X,b)=0$, where $p(b)=a$.
\end{theo}
He conjectures that (iii) holds for any $m$ (see \cite[Conjecture
5.1]{ogrady}). One of our results is the proof of O'Grady's
conjecture, see (i) below, and a generalization  to any degree, see
(ii) and (iii) below.
\begin{theo}\label{theointro2} Let $p:X\rightarrow Y$ be a degree $d$
finite morphism, where $X,\,Y$ are smooth projective and connected.

(i) Assume  $d=2$,  $a\in{\rm CH}_0(Y)$ is a $0$-cycle of degree $1$
supported on the branch locus of $p$, and $b:=\frac{1}{2}p^*a\in
{\rm CH}_0(X)$; if
 $\Gamma^m(Y,a)=0$, then
 $\Gamma^{2m-1}(X,b)=0$.

(ii)  For any $d$, assume $a\in Y$ is a point such that the subscheme $p^{-1}(a)$
 is supported on a point
$b\in X$.  If
 $\Gamma^m(Y,a)=0$, then
 $\Gamma^{d(m-1)+1}(X,b)=0$.

 (iii) For any $d$,  let  $b:=\frac{1}{d}p^*a$ for some
  $0$-cycle  $a\in{\rm CH}_0(Y)$ of degree $1$. If
  $\Gamma^m(Y,a)=0$ in ${\rm CH}_n(Y^m)/{\rm alg}$, then
  $\Gamma^{d(m-1)+1}(X,b)=0$ in ${\rm CH}_n(X^{d(m-1)+1})/{\rm alg}$.
\end{theo}
Statement (i) of Theorem \ref{theointro2}  has been obtained
independently by Moonen and Yin \cite{moonenyin}.
\begin{rema}\label{remaref}{\rm When $Y=\mathbb{P}^n$, and $d\leq n+1$,
there always exists a point $a\in \mathbb{P}^n$ as in (ii) (cf.
\cite{fula}). In this case, we have $\Gamma^m(Y,a)=0$, with $m=n+1$,
hence we conclude that for $d$-th covers $X$ of $\mathbb{P}^n$ with
$d\leq n+1$, $\Gamma^{dn+1}(X,b)=0$, with $b=\frac{1}{d}p^*(pt)$.
Note also that any curve $X$ of genus $g$ admits a morphism of
degree $d\leq g+1$ to $\mathbb{P}^1$, which is totally ramified  at
one given  point $x$. Hence we get $\Gamma^{g+2}(X,x)=0$ for any
$x\in X$. This last result is also proved by Moonen and Yin
\cite{moonenyin} using the Colombo-van Geemen vanishing result.}
\end{rema}

\begin{rema} \label{remanew25aout} {\rm In the case where $Y$ is $\mathbb{P}^1$, so $X$ is a $d$-gonal
curve, Theorem \ref{theointro2}, (iii) gives the vanishing
$\Gamma^{d+1}(X, b)=0$ in ${\rm CH}_1(X^{d+1})/{\rm alg}$. As explained
 in Subsection \ref{subsec}, this is equivalent to the
Colombo-van Geemen theorem \cite{colombovgeemen} mentioned above.}
\end{rema}

Another application of Theorem \ref{theointro2} is the following
result, which will be deduced from it in Section \ref{sec1} using
 the smash nilpotence result
of \cite{voe} for cycles algebraically equivalent to $0$:
\begin{coro}\label{corosec1} Let $X$ be a smooth projective(connected)  variety of dimension $n$. Then for any
$a\in  {\rm CH}_0( X)$ of degree
$1$, there exists an integer $m$ such that $\Gamma^m(X,a)=0$ in ${\rm CH}^{(m-1)n}(X^m)$.
\end{coro}

 Our second result is
the following more precise statement:
\begin{theo}\label{theointro1preciseintro}
Let $X$ be smooth projective connected of dimension $n$ and let
$a\in {\rm CH}_0( X)$ be of degree $1$. Then, if
 $X $ is swept-out by irreducible curves of
genus $g$ supporting a zero-cycle rationally equivalent to $a$, and
$m\geq (n+1)(g+1)$,
one has
$\Gamma^m(X,a)=0 $ in ${\rm CH}^{(m-1)n}(X^m)$.
\end{theo}
Note that such a $g$ always exists: Indeed, consider  curves in $X$
which are complete intersections of ample hypersurfaces containing
the support of the cycle $a$. For sufficiently high degree such
hypersurfaces, these curves will sweeep-out $X$, and thus we can
take for $g$ the genus of the generic such curves.
 In the case where $X$ is a $K3$ surface, we know that $X$ is swept-out
 by elliptic curves supporting the canonical $0$-cycle. Hence we get from Theorem
 \ref{theointro1preciseintro}
 the vanishing $\Gamma^6(X,o_X)=0$, which is not optimal in view
 of
the  relation (\ref{eqformk3}) in Theorem \ref{BV}.

We finally turn to the case of hyper-K\"ahler manifolds.
 For $K3$ surfaces, one can get as a consequence of (\ref{eqformk3})
the following properties of $o_X$ (note however that  property
\ref{itemi} below is used to prove (\ref{eqformk3}) so that we do
not actually recover it from (\ref{eqformk3}). Nevertheless, the
consequences \ref{itemi} and \ref{itemii} indicate that surfaces
satisfying (\ref{eqformk3}) are quite special):

\begin{enumerate}
\item \label{itemi}  The intersection of two divisors $D,D'$ on $X$ is proportional to $o_X$ in ${\rm CH}_0(X)$.
\item \label{itemii} The second Chern class $c_2(X)$ is equal to $24\, o_X$.
\end{enumerate}

In the paper \cite{ogrady}, O'Grady formulates the following generalization
of (\ref{eqformk3}):
\begin{conj}\label{conjogrady} (O'Grady, \cite[Conjecture 0.1]{ogrady}) Let $X$ be a hyper-K\"ahler $n$-fold.
Then there exists a canonical $0$-cycle $o_X\in {\rm CH}_0(X)$ of degree $1$ such that
$\Gamma^{n+1}(X,o_X)=0$ in  ${\rm CH}_n(X^{n+1})$.

\end{conj}
Note that by Theorem \ref{theoog}, (i), we have
$[\Gamma^{n+1}(X,o_X)]=0$ in $H^*(X^{n+1},\mathbb{Q})$ and that this
is optimal. Conjecture \ref{conjogrady} thus states that the cycles
$\Gamma^k(X,o_X)$ vanish in ${\rm CH}(X^k)$ once they vanish in
$H^*(X^k,\mathbb{Q})$, which is very different from the situation
encountered in the case of curves (except for the hyperelliptic
ones).

 O'Grady   establishes this conjecture for the punctual
Hilbert schemes $S^{[2]}$ and $S^{[3]}$ of a $K3$ surface. The
canonical $0$-cycle $o_X$, for $X=S^{[n]}$, is naturally defined as
the class in ${\rm CH}_0(X)$ of any point of $X$ lying over $no_S\in
S^{(n)}$, for some representative $o_S\in S$ of the canonical
$0$-cycle of $S$. We prove in section \ref{sec3} Conjecture
\ref{conjogrady} for punctual Hilbert schemes $X=S^{[n]}$ of $K3$
surfaces, and for any $n$,
   using methods from \cite{voisinpamq} and recent results of Yin \cite{yin}:
\begin{theo} \label{theoHKintro} Let $S$ be a $K3$ surface, and let $X=S^{[m]}$. Then
\begin{eqnarray}
\Gamma^{n+1}(X,o_X)=0 \,\,{\rm in \,\, CH}_n(X^{n+1}).
\label{eqformk3prouve}
\end{eqnarray}
where $o_X$ is the canonical $0$-cycle on $X$ coming from the
canonical $0$-cycle of $S$, and $n={\rm dim}\,X=2m$.
\end{theo}
Note that one can recover from (\ref{eqformk3prouve})  the following
result, which had been in fact already proved in \cite[Theorem 1.5]{voisinpamq}.
\begin{coro} The intersection of $n$ divisors on $X$ is proportional to
$o_X$ in ${\rm CH}_0(X)$.
\end{coro}
For the proof of Theorem \ref{theoHKintro}, we will need three
tools. The first  ingredient is  similar to what we did in
\cite{voisinpamq}, namely we will use the de Cataldo-Migliorini
theorem \cite{decami} and will prove Proposition
\ref{proppourtheofin} in order to reduce to computations in the Chow
rings of the self-products $S^{k}$. The second ingredient is very
new and it is provided by Yin's recent result \cite{yin} saying that
the cohomological relations between the big diagonals of a regular
surface and the pull-back of the class of a point are generated
(modulo trivial relations) by the pull-backs of the Kimura relation
and the cohomological counterpart $[\Gamma^3(S,o_S)]=0$ in
$H^8(S^3,\mathbb{Q})$ of the relation (\ref{eqformk3}) (see also
\cite[Proposition 1.3]{ogrady}). We then argue that the Kimura
relation is not needed in our context, while the relation
$\Gamma^3(S,o_S)=0$ is satisfied in the Chow ring by Theorem
\ref{BV}.

To conclude, let us remark that the following conjecture in the same
spirit as Conjecture \ref{conjogrady} was stated first in
\cite{voisinpamq} for $K3$ surfaces, and then in \cite{shenvial} for
general hyper-K\"ahler manifolds:
\begin{conj} \label{conjabove} Let $X$ be a projective
hyper-K\"ahler manifold and $n> 0$ be an integer. Then there exists
a canonical $0$-cycle $o_X\in{\rm CH}_0(X)$ such that any polynomial
relation between the cohomology classes $pr_i^*[o_X],\,i\leq
n,\,pr_{ij}^*[\Delta_X],\,i\not=j\leq n$, already holds in ${\rm
CH}(X^n)$.
\end{conj}
O'Grady's conjecture \ref{conjogrady} is the particular case of
Conjecture \ref{conjabove} which concerns the class
$\Gamma^{n+1}(X,o_X)$, $n={\rm dim}\,X$. As explained in
\cite{voisinbook} in the case of $K3$ surfaces, Conjecture
\ref{conjabove} is extremely strong since it implies finite
dimensionality in the Kimura sense, with very important consequences
established by Kimura \cite{kimura}, in particular on the nilpotency
of self-correspondences homologous to $0$. O'Grady's conjecture
\ref{conjogrady} does not seem to have such implications, so it is
possibly of a  nature different from Conjecture \ref{conjabove}.

 The paper is organized as follows: in
section \ref{sec1} we introduce  variants $\Gamma^{1,m}(X,a)$ of the
cycles $\Gamma^m(X,a)$ which lie in ${\rm CH}_n(X^{m+1})$, $n={\rm
dim}\,X$ and relate them to $\Gamma^m(X,a)$. In section \ref{sec2},
we will prove Theorem \ref{theointro1preciseintro}. Theorem
\ref{theointro2} will be proved in Section \ref{sec2bis} and Theorem
\ref{theoHKintro} will be proved in section \ref{sec3}. The last subsection \ref{secuniv} is devoted to the sketch of the proof of a general
theorem (Theorem
\ref{theopourtheofin}) concerning universally defined cycles on quasiprojective surfaces, which is used in the proof of Theorem \ref{theoHKintro}. This result
 is  of independent interest and its complete proof  will
be given together with further applications in \cite{voisinuniversalcycles}.

\vspace{0.5cm}

{\bf Thanks.} This paper has been completed at the Institute for Advanced Study, which I thank for ideal working conditions. I also thank Lie Fu, Ben Moonen, Kieran O'Grady, Burt Totaro
and  Qizheng Yin
 for interesting discussions related to this work, and the referee for his/her
useful suggestions and comments.
\section{Cycles $\Gamma^{1,m}(X,a)$ \label{sec1}} We first introduce the
following notation: $X$ being smooth projective, and $a\in{\rm
CH}_0(X)_\mathbb{Q}$ being a zero-cycle of degree $1$, we define
$\Gamma^{1,m}(X,a)\in {\rm CH}_n(X^{m+1})_\mathbb{Q}$ by
\begin{eqnarray}\label{formulaforgamma1m}
\Gamma^{1,m}(X,a):=\prod_{1\leq i\leq
m}({p_{0i}}^*\Delta_X-{p_i}^*a),
\end{eqnarray}
where \begin{enumerate} \item $\Delta_X\subset X\times X$ is the
diagonal of $X$, \item $p_{0i}:X^{m+1}\rightarrow X\times X$ is the
projection on the product of the first and $i+1$-th factors,
\item
$p_i:X^{m+1}\rightarrow X$ is the projection on the $i+1$-th factor
(our factors are indexed by $\{0,\ldots,m\}$).

\end{enumerate}
The cycles $\Gamma^m(X,a)$ and $\Gamma^{1,m}(X,a)$ are related as
follows:
\begin{lemm}\label{legammamgamma1} We have the following formula:
\begin{eqnarray}\label{eqformproj}
\Gamma^m(X,a)=p'_{1,\ldots,m*}\Gamma^{1,m}(X,a),
\end{eqnarray}
where we index the factors of $X^{m+1}$ by $\{0,\ldots,m\}$ and
$p'_{1,\ldots,m}$ is the  projection from $X^{m+1}$ to the products
$X^m$ of its last $m$ factors. We also have:
\begin{eqnarray}\label{eqformrec}
\Gamma^{m+1}(X,a)=\Gamma^{1,m}(X,a)-p_0^*a\cdot
{p'_{1,\ldots,m}}^*(\Gamma^m(X,a)).
\end{eqnarray}
\end{lemm}
\begin{proof} This is almost immediate. Developing the product in
(\ref{formulaforgamma1m}), we get
\begin{eqnarray}\label{eqformrec111}\Gamma^{1,m}(X,a)=\sum_{I\subset
\{1,\ldots,m\},|I|=i}(-1)^i{p'_I}^*(a^{*i})\cdot
p_{0,J}^*\Delta_{m+1-i}, \end{eqnarray}
 where $I\bigsqcup J=\{1,\dots,
m\}$, $p_{0,J}$ is the projection from $X^{m+1}$ to the product
$X^{m+1-i}$ of factors indexed by $\{0\}\cup J$ and $p'_I$ is the
projection from $X^{m+1}$ to the product $X^i$ of the factors
indexed by $I\subset \{1,\ldots,m\}$. Applying
$p'_{1,\ldots,m*}:{\rm CH}_n(X^{m+1})_\mathbb{Q}\rightarrow {\rm
CH}_n(X^{m})_\mathbb{Q}$, we get by the projection formula, using
the fact that $p'_I=p_I\circ p'_{1,\ldots,m}$:
$$p'_{1,\ldots,m*}\Gamma^{1,m}(X,a)=\sum_{I\subset \{1,\ldots,m\},|I|=i}(-1)^ip_I^*(a^{*i})\cdot
p'_{1,\ldots,m*}(p_{0,J}^*\Delta_{m+1-i}).$$ Formula
(\ref{eqformproj}) then follows from the fact that
$p_J^*\Delta_{m-i}=p'_{1,\ldots,m*}(p_{0,J}^*\Delta_{m+1-i})$ in
${\rm CH}_n(X^{m})$.

As for (\ref{eqformrec}), we first write formula
(\ref{eqformgammam}) for $X^{m+1}$, where as above we index the
factors of $X^{m+1}$ by $\{0,\ldots,m\}$. This gives us
\begin{eqnarray} \label{eqformgammam+1}
\Gamma^{m+1}(X,a)=\sum_{I\subset \{0,\ldots,m\},i=|I|\leq
m}(-1)^i{p'_I}^*(a^{*i})\cdot {p'_J}^*\Delta_{m+1-i}\in {\rm
CH}_n(X^{m+1})_\mathbb{Q}. \end{eqnarray} We now separate the terms
where $0\not\in I$, which by (\ref{eqformrec111}) exactly give
$\Gamma^{1,m}(X,a)$, and the terms where $0\in I$, which exactly
give $-p_0^*a\cdot {p'_{1,\ldots,m}}^*(\Gamma^m(X,a))$.
\end{proof}
We deduce the following
\begin{prop}\label{propeq} The vanishing of $\Gamma^m(X,a)$ in ${\rm CH}_n(X^m)$ is equivalent to the
vanishing of  $\Gamma^{1,m}(X,a)$ in  ${\rm CH}_n(X^{m+1})$.
\end{prop}
\begin{proof} Indeed, if $\Gamma^{1,m}(X,a)=0$ then
$\Gamma^m(X,a)=0$ by (\ref{eqformproj}). Conversely, if
$\Gamma^m(X,a)=0$, then \cite[Proposition 2.4]{ogrady} shows that
also $\Gamma^{m+1}(X,a)=0$. Formula (\ref{eqformrec}) then implies
that $\Gamma^{1,m}(X,a)=0$.
\end{proof}
A consequence of this result is the following statement comparing
$\Gamma^m(X,a)$ and $\Gamma^m(X,b)$, for two $0$-cycles
$a,\,b\in{\rm CH}_0(X)$ of degree $1$.
\begin{coro} If $\Gamma^m(X,a)=0$ and the cycle $b-a$ satisfies $(b-a)^{*k}=0$ in
${\rm CH}_0(X^k)$, then $\Gamma^{m+k}(X,b)=0$.
\end{coro}
Here we refer to (\ref{eqstar}) for the definition of the $*$-product (or external product) of cycles.
\begin{proof} Indeed, by Proposition \ref{propeq}, the assumption is equivalent to
the vanishing  conditions:
$$\Gamma^{1,m}(X,a)=\prod_{i=1}^{i=m}(p_{0i}^*(\Delta_X)-p_i^*a)=0\,\,{\rm in}\,\, {\rm CH}_n(X^{m+1}),
 \,\,n={\rm dim}\,X$$
 $$\prod_{i=1}^{i=k}p_{i}^*(b-a)=0\,\,{\rm in}\,\,{\rm CH}_0(X^k).$$
On the other hand, the conclusion is equivalent to the vanishing
$$\Gamma^{1,m+k}(X,a)=\prod_{i=1}^{m+k}(p_{0i}^*(\Delta_X)-p_i^*b)=0\,\,{\rm in}\,\,
{\rm CH}_n(X^{m+k+1}).$$
 We now write $b=a+(b-a)$, getting
$$\Gamma^{1,m+k}(X,b)=\prod_{i=1}^{m+k}((p_{0i}^*(\Delta_X)-p_i^*a)-p_i^*(b-a))$$
and develop the product.
In the developed expression, the product of
$\geq m$ terms of the form $p_{0i}^*(\Delta_X)-p_i^*a$ is $0$ and the product of
$\geq k$ terms of the form $p_i^*(b-a)$ is $0$. Hence we conclude that
each monomial in the development is $0$.

\end{proof}
Here is another corollary of Proposition \ref{propeq}. It shows how to deduce Corollary \ref{corosec1} from Theorem
\ref{theointro2}, and thus gives  another proof of the nilpotency statement
of Theorem \ref{theointro1preciseintro}, with no estimate on the nilpotency index.
\begin{coro} Let $X$ be a smooth projective connected variety and let $a$ be a $0$-cycle
of degree $1$  on $X$ such that $\Gamma^m(X,a)=0$ in ${\rm
CH}(X^m)/{\rm alg}$. Then for any $0$-cycle $b$ of degree $1$ on
$X$, there is an integer $M$ such that $\Gamma^M(X,b)=0$ in ${\rm
CH}(X^M)$.
\end{coro}
\begin{proof} As $a$ and $b$ are algebraically equivalent, we also have
$\Gamma^m(X,b)=0$ in ${\rm CH}(X^m)/{\rm alg}$. By Proposition
\ref{propeq}, which is true and proved in the same way for cycles
modulo algebraic equivalence (observing that \cite[Proposition
2.4]{ogrady} is true as well for cycles modulo algebraic
equivalence), this is equivalent to the fact that
$\Gamma^{1,m}(X,b)$ is algebraically equivalent to $0$ in $X^{m+1}$.
By the smash-nilpotence  result of Voevodsky \cite{voe}, there is an
integer $N$ such that the cycle $\Gamma^{1,m}(X,b)^{*N}$ vanishes
identically in ${\rm CH}(X^{N(m+1)})$. Thus its restriction to
$X^{Nm+1}$ embedded in $X^{N(m+1)}$ by the small diagonal on the
factors of index $0,m+1,2m+1,\ldots,(N-1)m+1$ also vanishes in ${\rm
CH}(X^{Nm+1})$. But this restricted cycle is nothing but
$\Gamma^{1,Nm}(X,b)$.

\end{proof}
The following criterion for the vanishing of $\Gamma^m(X,a)$ will be
used in Section \ref{sec2bis}. Here we consider more generally the
vanishing of $\Gamma^m(X,a)$ modulo an adequate equivalence relation
$R$ which in applications will be rational or algebraic equivalence.
We need an assumption on the $0$-cycle $a$ of degree $1$, namely
\begin{eqnarray}\label{eqcyclea}
 p_1^*a\cdot p_2^*a=\Delta_*a \,\,{\rm in}\,\,{\rm CH}_0(X\times X)/R,
 \end{eqnarray}
 where $\Delta$ is the diagonal inclusion map  of $X$ in $X\times X$.
 This assumption is satisfied for any $R$
 if $a$ is a point, or for any $0$-cycle if $R$ is algebraic  equivalence, and $X$ is connected.
 \begin{prop}\label{propourdm1}
  Assume $a$ satisfies (\ref{eqcyclea}). Then
 $\Gamma^m(X,a)=0$ in ${\rm CH}(X^m)/R$ if and only if
 $$\Gamma^{1,m-1}(X,a)=p_0^*a\cdot \Gamma\,\,{\rm in}\,\,{\rm CH}_n(X^m)/R,\,n={\rm dim}\,X$$
 for some cycle $\Gamma\in {\rm CH}_{2n}(X^m)/R$.
 \end{prop}
The proof of Proposition \ref{propourdm1} will use the following
\begin{lemm}\label{lepourprop22avril} Assume the degree $1$ zero-cycle $a$ of $X$
satisfies (\ref{eqcyclea}). Then for any
$Y$ and any cycle $\Gamma\in {\rm CH}(X\times Y)/R$, the following formula holds:
$$p_X^*a\cdot \Gamma=p_X^*a\cdot p_Y^*\Gamma_a \,\,{\rm in}\,\,{\rm
CH}(X\times Y)/R,$$ where
$$\Gamma_a:=p_{Y*}(p_X^*a\cdot \Gamma)\in {\rm CH}(Y)/R.$$
\end{lemm}
\begin{proof} Let $a=\sum_in_ia_i$, where $a_i\in X$. Then
$$p_X^*a\cdot \Gamma=\sum_in_ip_X^*a_i\cdot \Gamma=\sum_in_i a_i\times \Gamma_{a_i}=\sum_in_ip_X^*a_i\cdot
p_Y^*\Gamma_{a_i},$$ where $\Gamma_{a_i}\in {\rm CH}(Y)/R$ is the
restriction of $\Gamma$ to $a_i\times Y$. So we need to prove that,
assuming (\ref{eqcyclea}),
\begin{eqnarray}\label{eqpouraa22avril} \sum_in_ip_X^*a_i\cdot
p_Y^*\Gamma_{a_i}=p_X^*a\cdot p_Y^*\Gamma_a\,\,{\rm in}\,\,{\rm
CH}(X\times Y)/R,
\end{eqnarray}
where $\Gamma_a=\sum_in_i\Gamma_{a_i}\in {\rm CH}(Y)/R$. Note that
(\ref{eqcyclea}) is exactly the case of (\ref{eqpouraa22avril})
where $X=Y$ and $\Gamma$ is the diagonal of $X$. The general  case
is then deduced from this one by introducing $X\times X\times Y$
with its various projections to $X, \,Y$ and $X\times Y$. Namely,
let \begin{itemize} \item $p_{i,X,Y}:X\times X\times Y\rightarrow
X\times Y$, $i=1,\,2$, be the projections onto  the product of the
first and the third (resp. the second and the third) factor,
\item
$p_{i,X}:X\times X\times Y\rightarrow X,\,i=1,\,2$ be the projection
on the
 first, resp. second, factor
and \item $p_{X,X}:X\times X\times Y\rightarrow X\times Y\rightarrow
X\times X$ be the projection onto  the product of the first two
factors. \end{itemize}
The left hand side of (\ref{eqpouraa22avril})
is clearly equal to
$$p_{2,X,Y*}(p_{1,X}^*a\cdot p_{X,X}^*\Delta_X\cdot p_{2,X,Y}^*\Gamma).$$
Formula (\ref{eqcyclea}) tells us that
on $X\times X$, $p_1^*a\cdot\Delta_X=p_1^*a\cdot p_2^*a$ modulo $R$,
so that
\begin{eqnarray}\label{autreeqpouraa22avril}p_{2,X,Y*}(p_{1,X}^*a\cdot p_{X,X}^*\Delta_X\cdot
 p_{2,X,Y}^*\Gamma)=
p_{2,X,Y*}(p_{1,X}^*a\cdot p_{2,X}^*a\cdot p_{2,X,Y}^*\Gamma)
\,\,{\rm in}\,\,{\rm CH}(X\times Y)/R.
\end{eqnarray}
As $a$ has degree $1$, the right hand side of
(\ref{autreeqpouraa22avril}) is equal by the projection formula to
$$p_X^*a\cdot p_Y^*\Gamma_a,$$
proving (\ref{eqpouraa22avril}).

\end{proof}
\begin{proof}[Proof of Proposition \ref{propourdm1}] We have by (\ref{eqformrec})
$$\Gamma^m(X,a)=\Gamma^{1,m-1}(X,a)-p_0^*a\cdot p_{1,\ldots,m-1}^*\Gamma^{m-1}(X,a)$$
so if $\Gamma^m(X,a)=0$ in ${\rm CH}(X^m)/R$, we get
$$\Gamma^{1,m-1}(X,a)=p_0^*a\cdot p_{1,\ldots,m-1}^*\Gamma^{m-1}(X,a)\,\,{\rm in }\,\,{\rm CH}(X^m)/R.$$
This proves one direction (for which we do not need
(\ref{eqcyclea})). In the other direction, we assume
(\ref{eqcyclea}) and
\begin{eqnarray}\label{autreeqpouraa22avril111}
\Gamma^{1,m-1}(X,a)=p_0^*a\cdot \Gamma\,\,{\rm in}\,\,{\rm
CH}_n(X^m)/R
\end{eqnarray}
 for some cycle $\Gamma\in {\rm CH}_{2n}(X^m)/R$.
 We now use Lemma \ref{lepourprop22avril} which gives
$$p_0^*a\cdot \Gamma=p_0^*a\cdot p_{1,\ldots,m-1}^*(p_{1,\ldots,m-1*}(p_0^*a\cdot \Gamma)).$$
By (\ref{autreeqpouraa22avril111}), this gives
$$p_0^*a\cdot \Gamma=p_0^*a\cdot p_{1,\ldots,m-1}^*(p_{1,\ldots,m-1*}(\Gamma^{1,m-1}(X,a)))
\,\,{\rm in}\,\,{\rm CH}_n(X^m)/R.$$
As $p_{1,\ldots,m-1*}(\Gamma^{1,m-1}(X,a))=\Gamma^{m-1}(X,a)$ by
(\ref{eqformproj}), we get
$$\Gamma^{1,m-1}(X,a)=p_0^*a\cdot p_{1,\ldots,m-1}^*(\Gamma^{m-1}(X,a))\,\,{\rm in}\,\,{\rm CH}_n(X^m)/R.$$
Using (\ref{eqformrec}), we conclude that
$$\Gamma^m(X,a)=\Gamma^{1,m-1}(X,a)-p_0^*a\cdot
p_{1,\ldots,m-1}^*(\Gamma^{m-1}(X,a))=0\,\,{\rm in}\,\,{\rm
CH}_n(X^m)/R.$$
\end{proof}

\section{Proof of Theorem \ref{theointro1preciseintro} \label{sec2}}
We prove in this section  Theorem \ref{theointro1preciseintro}, that
is the following statement :
\begin{theo}\label{theo1precise} Let $X$ be a variety of dimension $n$ and let $a\in {\rm CH}_0(X)$
be of degree $1$.
If $X $ is swept-out by irreducible curves  of genus $\leq g$
supporting a $0$-cycle rationally equivalent to $a$, and $m\geq
(n+1)(g+1)$, then $\Gamma^m(X,a)=0 $.
\end{theo}
Note that for $g=0$, we get the following corollary:
\begin{coro} \label{coroRC} Let $X$ be a rationally connected manifold of dimension $n$.
Then $\Gamma^{n+1}(X,o)=0$ for any point $o\in X$.
\end{coro}
This corollary will be improved at the end of this section in
Theorem \ref{theovariantRC}.
\begin{proof}[Proof of Theorem \ref{theo1precise}] By Proposition
\ref{propeq}, it suffices to prove the vanishing of
$\Gamma^{1,m}(X,a)$. Let us see $\Gamma^{1,m}(X,a)$ as a
correspondence between $X$ and $X^m$. Then for any $x\in X$, we have
$$ \Gamma^{1,m}(X,a)_{\mid x\times X^m}=(x-a)^{*m}\,\,{\rm
in}\,\,{\rm CH}_0(X^{m})_\mathbb{Q}.$$

Recall now the following result proved in \cite{voe}, \cite{voivoe}:
\begin{lemm} \label{levoe} Let $C$ be a smooth connected curve of genus $g$, and let
$z\in {\rm CH}_0(C)_\mathbb{Q}$ be a $0$-cycle of degree $0$ on $C$.
Then for $k>g$, $z^{*k}=0$ in ${\rm CH}_0(C^k)_\mathbb{Q}$.

\end{lemm}

Our assumption is now that $X$ is swept out by  irreducible curves
of genus $\leq g$ supporting a $0$-cycle rationally equivalent to
$a$. This means that for any $x\in X$, there is a smooth connected
curve $C_x$ of genus $\leq g$ mapping to $X$ via a morphism $f_x$, a
point $x'\in C_x$ such that $f_x(x')=x$ and a $0$-cycle $a'\in {\rm
CH}_0(C_x)_\mathbb{Q}$ of degree $1$, such that $f_{x*}(a')=a$ in
${\rm CH}_0(X)_\mathbb{Q}$. It is then clear that
$${f_x^k}_*((x'-a')^{*k})=(x-a)^{*k}\,\,{\rm in}\,\,{\rm
CH}_0(X^k)_\mathbb{Q}.$$ We thus conclude by Lemma \ref{levoe} that
for $k>g$, and for any $x\in X$
\begin{eqnarray} \label{eq0k}
\Gamma^{1,k}(X,a)_{\mid x\times X^k}=(x-a)^{*k}=0\,\,{\rm
in}\,\,{\rm CH}_0(X^k)_\mathbb{Q}.
\end{eqnarray}

 We use now the following
general principle which is behind
 the  Bloch-Srinivas decomposition of the diagonal
\cite{blochsrinivas}, see  \cite[3.1]{voisinbook}:
\begin{theo} Let $\phi:W\rightarrow Y$ be a morphism, where $W$
is smooth of dimension $m$. Let $Z$ be a codimension $k$ cycle on
$W$. Assume that, for general $y\in Y$, the restriction $Z_{\mid
W_y}$ vanishes in ${\rm CH}^k(W_y)$. Then there is a dense Zariski
open set $U\subset Y$, such that $Z_U=0$ in ${\rm CH}^k(W_U)$.
Equivalently, there exist a nowhere dense closed algebraic
 subset $D\subsetneqq Y$ and a cycle $Z'\in {\rm
 CH}_{m-k}(W_D)_\mathbb{Q}$ such that
 $$Z=Z'\,\,{\rm in\,\,CH}^k(W)_\mathbb{Q}.$$
\end{theo}
(Here we use the notation $W_D:=\phi^{-1}(D),\,W_U:=\phi^{-1}(U)$.)
Applying this statement to $Y=X,\,W=X^{k+1},\,\phi$ the projection
to the first factor and $Z=\Gamma^{1,k}(X,a)$, we conclude from
(\ref{eq0k}) that under the assumptions of Theorem
\ref{theo1precise}, there exists for $k>g$ a proper closed algebraic
subset $D\subsetneqq X$, such that $\Gamma^{1,k}(X,a)$ is rationally
equivalent to a cycle supported on $D\times X^k$.

Recall now the formula (\ref{formulaforgamma1m}) defining
$\Gamma^{1,k}$:
$$\Gamma^{1,k}(X,a):=\prod_{1\leq i\leq
k}(p_{0i}^*\Delta_X-p_i^*a).
$$
It follows immediately that
\begin{eqnarray}
\label{eqkk'} \Gamma^{1,k+k'}(X,a)=p_{0,1\leq i\leq
k}^*\Gamma^{1,k}(X,a)\cdot p_{0,k+1\leq i\leq
k+k'}^*\Gamma^{1,k'}(X,a), \end{eqnarray}
 where $$p_{0,1\leq i\leq
k}:X^{k+k'+1}\rightarrow X^{k+1}$$ is the projection on the product
of the $k+1$ first factors and
$$p_{0,k+1\leq i\leq
k+k'}:X^{k+k'+1}\rightarrow X^{k'+1}$$ is the projection on the
product of the first factor (indexed by $0$) and the last $k'$
factors.

For $m\geq (n+1)(g+1)$, we write $m=(n+1)(g+1)+r$, for some $r\geq
0$ and we get from (\ref{eqkk'}):
$$\Gamma^{1,m}(X,a)=
p_{0,1\leq i\leq g+1}^*(\Gamma^{1,g+1})\cdot p_{0, g+2\leq i\leq
2(g+1)}^*(\Gamma^{1,g+1})$$ $$\ldots p_{0,n(g+1)+1\leq i\leq
(n+1)(g+1)}^*(\Gamma^{1,g+1})\cdot p_{0,(n+1)(g+1)+1\leq i\leq
(n+1)(g+1)+r}^*(\Gamma^{1,r}) .$$
 Now we proved that the cycle $\Gamma^{1,g+1}$ is supported (via
the first projection $X^{g+2}\rightarrow X$) over a proper algebraic
subset $D\varsubsetneqq X$, and by the easy moving Lemma
\ref{lemoving} below, we can choose closed algebraic subsets
$D_1,\ldots, D_{n+1}$ such that $\cap_iD_i=\emptyset$ and
$\Gamma^{1,g+1}$ is supported (via the first projection
$X^{g+2}\rightarrow X$) over the proper algebraic subset
$D_i\varsubsetneqq X$ for each $i$.

Then we conclude that for $m\geq (n+1)(g+1)$, $\Gamma^{1,m}(X,a)$ is
supported (via the first projection $X^{(n+1)(g+1)+r+1}\rightarrow
X$) over the proper algebraic subset $\cap_iD_i=\emptyset$, and thus
is equal to $0$.
\end{proof}
\begin{lemm}\label{lemoving} Let $Y$ be irreducible and let $Z$ be a cycle on a product
$Y\times W$. Assume there exists a proper closed algebraic subset
$D\varsubsetneqq Y$ such that $Z$ is rationally equivalent to a
cycle $Z'$ supported on $D\times W$. Then for any finite set of
points $y_1,\ldots,y_l\in Y$, there is a $D'\varsubsetneqq Y$ such
that none of the $y_j$'s belongs to $D'$ and $Z$ is rationally
equivalent to a cycle $Z''$ supported on $D'\times W$.
\end{lemm}
\begin{proof}  Let $\tau:\widetilde{D}\rightarrow D$
 be a desingularization of $D\stackrel{i}{\hookrightarrow}Y$.
 The cycle $Z'$ of $D\times W$ with rational coefficients  lifts to a cycle
$\widetilde{Z}'$ of $\widetilde{D}\times W$. Let $\tilde{i}=i\circ
\tau:\widetilde{D}\rightarrow Y$ be the natural map and let
$\Gamma_{\tilde{i}}\subset \widetilde{D}\times Y$ be its graph.
Since $\Gamma_{\tilde{i}}$ has codimension $n={\rm dim}\,Y$, and
dimension $\leq n-1$, there is a cycle $\Gamma'\subset
\widetilde{D}\times Y$ rationally equivalent to $\Gamma_{\tilde{i}}$
and not intersecting $\widetilde{D}\times \{y_1,\ldots,y_l\}$. In
other words, $pr_2({\rm Supp}\,\Gamma')$ does not contain any of the
points $y_i$. We have by assumption
$$Z=({i},Id_W)_*Z'=(\tilde{i},Id_W)_*\widetilde{Z}'=(\Gamma_{\tilde{i}},Id_W)_*(\widetilde{Z}')$$
$$=(\Gamma',Id_W)_*(\widetilde{Z}')$$
in ${\rm CH}(Y\times W)$. Now, the cycle
$(\Gamma',Id_W)_*(\widetilde{Z}')$ is supported on $pr_2({\rm
Supp}\,\Gamma')\times W$, so the result is proved with $D'=pr_2({\rm
Supp}\,\Gamma')$, and $Z''=(\Gamma',Id_W)_*(\widetilde{Z}')$.

\end{proof}
To conclude this section, let us observe that the same scheme of proof proof applies to give the following
result, which is a generalization of Corollary \ref{coroRC}:
\begin{theo} \label{theovariantRC} Let $X$ be a connected smooth projective variety with ${\rm CH}_0(X)=\mathbb{Z}$.
Then for the canonical degree $1$
$0$-cycle $o$ on $X$, $\Gamma^{n+1}(X,o)=0$ in ${\rm CH}_n(X^{n+1})$, where $n={\rm dim}\,X$.
\end{theo}
\begin{proof} Indeed, the Bloch-Srinivas decomposition of the diagonal \cite{blochsrinivas}
gives an equality
$$\Delta_X-X\times o=Z\,\,{\rm in}\,\,{\rm CH}_n(X\times X),$$
where $Z$ is supported over $D\times X$, for some divisor $D\subset X$.
By Lemma \ref{lemoving}, we can write such a decomposition with
$n+1$ divisors $D_1,\ldots,D_{n+1}$ such that $\cap_iD_i=\emptyset$.
We then conclude that $\Gamma^{1,n+1}(X,o)=\prod_{i=1}^{n+1}p_{0i}^*(\Delta_X-X\times o)$ is equal to $0$ in
${\rm CH}_n(X^{n+2})$, and  it follows from
Proposition \ref{propeq} that $\Gamma^{n+1}(X,o)=0$ in
${\rm CH}_n(X^{n+1})$.
\end{proof}
\section{Proof of Theorem \ref{theointro2}\label{sec2bis}}
We will first give the proof of Theorem \ref{theointro2}, (i). Let us recall the statement:

\begin{theo}\label{theogradyconj} Let $Y$ be smooth projective, and let $\pi:X\rightarrow Y$ be a
degree $2$ finite morphism, where $X$ is smooth projective. Let
$a\in {\rm CH}_0(Y)$ be a $0$-cycle of degree $1$ supported on the
branch locus of $\pi$. Then if $\Gamma^m(Y,a)=0$, we have
$\Gamma^{2m-1}(X,b)=0$, where $b=\frac{1}{2}\pi^*a\in {\rm
CH}_0(X)$.
\end{theo}\begin{rema}{\rm The assumption made on $a$ and $b$ is maybe not
optimal, but in any case the condition $b=\frac{1}{2}\pi^*a$ is not
sufficient. Indeed, consider the case where $Y$ is connected with
$\Gamma^m(Y,a)=0$, and $X$ consists of two copies of $Y$  with
$b=\frac{1}{2}\pi^*a\in {\rm CH}_0(X)$. Then $\Gamma^{k}(X,b)$ is
different from $0$ for any $k$ (in fact it is not even cohomologous
to $0$).}
\end{rema}
 We will
denote by $\pi_2=(\pi,\pi):X\times X\rightarrow Y\times Y$. Let
$i:X\rightarrow X$ be the involution of $X$ over $Y$ and
$\Gamma_i\subset X\times X$ be its graph. We then have
$$\pi_2^*(\Delta_Y)=\Delta_X+\Gamma_i.$$
Let
$$\Delta_X^+=\pi_2^*(\Delta_Y)=\Delta_X+\Gamma_i,\,\,\Delta_X^-=\Delta_X-\Gamma_i.$$
We thus have \begin{eqnarray} \label{eqpetiteform}
2\Delta_X=\Delta_X^++\Delta_X^-.
\end{eqnarray}

\begin{lemm} \label{leutilepour23mars} Under the assumptions of Theorem \ref{theogradyconj}, we have
the following
equalities in ${\rm CH}_n(X\times X\times X)$, $n:={\rm dim}\,X$.
\begin{eqnarray}\label{eqegacruc}
p_{12}^*\Delta_X^-\cdot p_{13}^*\Delta_X^-=p_{12}^*\Delta_X^+\cdot
p_{23}^*\Delta_X^-,
\\
p_2^*b\cdot p_{23}^*\Delta_X^-=0, \label{eqegacrucab}
\end{eqnarray}
hence
\begin{eqnarray}\label{uneeqpourd2} p_{12}^*\Delta_X^-\cdot p_{13}^*\Delta_X^-=
p_{12}^*(\Delta_X^+-2p_2^*b)\cdot
p_{23}^*\Delta_X^-.
\end{eqnarray}
\end{lemm}
\begin{proof} We compute the left hand side of (\ref{eqegacruc}); we have:
$$p_{12}^*\Delta_X^-\cdot
p_{13}^*\Delta_X^-=p_{12}^*(\Delta_X-\Gamma_i)\cdot
p_{13}^*(\Delta_X-\Gamma_i)$$
$$=p_{12}^*\Delta_X\cdot p_{13}^*\Delta_X-p_{12}^*\Delta_X\cdot
p_{13}^*\Gamma_i-p_{12}^*\Gamma_i\cdot
p_{13}^*\Delta_X+p_{12}^*\Gamma_i\cdot p_{13}^*\Gamma_i.$$ We
observe now that
$$p_{12}^*\Delta_X\cdot p_{13}^*\Delta_X=p_{12}^*\Delta_X\cdot
p_{23}^*\Delta_X,\,\,\, p_{12}^*\Delta_X\cdot
p_{13}^*\Gamma_i=p_{12}^*\Delta_X\cdot p_{23}^*\Gamma_i,$$
$$p_{12}^*\Gamma_i\cdot
p_{13}^*\Delta_X=p_{12}^*\Gamma_i\cdot
p_{23}^*\Gamma_i,\,\,\,p_{12}^*\Gamma_i\cdot
p_{13}^*\Gamma_i=p_{12}^*\Gamma_i\cdot p_{23}^*\Delta_X.$$ It thus
follows that
\begin{eqnarray}\label{eqn1} p_{12}^*\Delta^-\cdot p_{13}^*\Delta^-=p_{12}^*\Delta_X\cdot
p_{23}^*\Delta_X-p_{12}^*\Delta_X\cdot
p_{23}^*\Gamma_i-p_{12}^*\Gamma_i\cdot
p_{23}^*\Gamma_i+p_{12}^*\Gamma_i\cdot p_{23}^*\Delta_X.
\end{eqnarray}
The right hand side of (\ref{eqn1}) is clearly equal to
$$(p_{12}^*\Delta_X +p_{12}^*\Gamma_i    )\cdot(p_{23}^*\Delta_X-
p_{23}^*\Gamma_i),
$$
which is by definition $p_{12}^*\Delta_X^+\cdot p_{23}^*\Delta_X^-$,
thus proving formula (\ref{eqegacruc}).

In order to prove formula (\ref{eqegacrucab}), we use the fact that
the $0$-cycle $b$ can be written as $\sum_jn_j x_j$, where the
$x_j$'s are $i$-invariant. By linearity, it thus suffices to prove
(\ref{eqegacrucab}) when $b$ is an $i$-invariant point of $X$. Now
we have
$$p_2^*b\cdot p_{23}^*\Delta^-=p_{23}^*((b,b)-(b,ib))=0.$$
\end{proof}

\begin{proof}[Proof of Theorem \ref{theogradyconj}]
 By (\ref{eqformproj}), we have to prove that
$$p_{1,\ldots,2m-1*}(\Gamma^{1,2m-1}(X,b))=0\,\,{\rm in}\,\,{\rm CH}_n(X^{2m-1})_\mathbb{Q}.$$ Now, by
(\ref{formulaforgamma1m}) and (\ref{eqpetiteform}),  using
$$2b=\pi^*a,\,\,\Delta_X^+=\pi_2^*\Delta_Y,$$
 we get:
\begin{eqnarray}\label{eqnpourtheo1}
2^{2m-1}\Gamma^{1,2m-1}(X,b)=p_{01}^*(\pi_2^*\Delta_Y^a+\Delta_X^-)\cdot
\ldots \cdot p_{0,2m-1}^*(\pi_2^*\Delta_Y^a+\Delta_X^-).
\end{eqnarray}
Here we use the notation
$$\Delta_Y^a=\Delta_Y-p_2^*a\in {\rm CH}_n(Y\times Y)_\mathbb{Q},\,\,$$
so that we have $\Delta_X^+-2p_2^*b=\pi_2^*\Delta_Y^a$ and
(\ref{uneeqpourd2}) can be written as
\begin{eqnarray}\label{uneeqpourd2bis} p_{12}^*\Delta_X^-\cdot p_{13}^*\Delta_X^-=
p_{12}^*(\pi_2^*\Delta_Y^a)\cdot
p_{23}^*\Delta_X^-.
\end{eqnarray}
our assumption $\Gamma^m(Y,a)=0$ on $Y$ can be written using
Proposition \ref{propeq} as
\begin{eqnarray}\label{eqhyp}q_{01}^*\Delta_Y^a\cdot q_{02}^*\Delta_Y^a\cdot\ldots\cdot
q_{0m}^*\Delta_Y^a=0\,\,{\rm in}\,\,{\rm CH}_n(Y^{m+1})_\mathbb{Q},
\end{eqnarray}
where the $q_{0i}:Y^{m+1}\rightarrow Y\times Y$ are the projectors
onto the product of the first and  $i+1$-th factors.

Denote by $\pi_r:X^r\rightarrow Y^r$. We then clearly have for any
$r$
\begin{eqnarray}
\label{eqntire} \pi_{r+1}^*(q_{01}^*\Delta_Y^a \cdot\ldots\cdot
q_{0r}^*\Delta_Y^a)=p_{01}^*(\pi_2^*\Delta_Y^a)\cdot \ldots \cdot
p_{0r}^*(\pi_2^*\Delta_Y^a)\,\,{\rm in}\,\,{\rm CH}(X^r),
\end{eqnarray}
 and similarly for any
choice of indices $i_1,\ldots,i_r$ in $\{1,\ldots,2m-1\}$.
Developing now the product in (\ref{eqnpourtheo1}), we get a sum of
monomials which up to reordering  the factors, take the form
\begin{eqnarray}
\label{eqntire1} p_{01}^*(\pi_2^*\Delta_Y^a)\cdot \ldots \cdot
p_{0r}^*(\pi_2^*\Delta_Y^a)\cdot
p_{0,r+1}^*\Delta_X^-\cdot\ldots\cdot p_{0,2m-1}^*\Delta_X^-
\end{eqnarray}
for some $r$.
 These
terms vanish for $r\geq m$ by (\ref{eqntire}) and (\ref{eqhyp}).

We now conclude the proof as follows: The terms
$p_{0i}^*\Delta_{X}^-$ for $i\geq r+1$ can be grouped by pairs, and
there are at least $\llcorner\frac{2m-1-r}{2}\lrcorner$ such pairs.
  By (\ref{uneeqpourd2bis}), for each such pair, we have
$$p_{0i}^*\Delta_{X}^-\cdot p_{0,i+1}^*\Delta_{X}^-=p_{0i}^*(\pi_2^*\Delta_Y^a)\cdot
p_{i,i+1}^*\Delta^-.$$ Hence each such pair produces a summand
$p_{0i}^*(\pi_2^*\Delta_Y^a)$. In total we get in (\ref{eqntire1})
at least $r+\llcorner\frac{2m-1-r}{2}\lrcorner$ factors of the form
$p_{0j}^*(\pi_2^*\Delta_Y^a)$. Now we have
$r+\llcorner\frac{2m-1-r}{2}\lrcorner\geq m$ unless $r=0$, and it
follows that (\ref{eqntire1}) vanishes for $r\geq1$.  Hence we
proved that the only possibly nonzero monomial of the form
(\ref{eqntire1}) in the developed product (\ref{eqnpourtheo1}) is
 $p_{01}^*(\Delta_X^-)\cdot \ldots \cdot p_{0,2m-1}^*\Delta_X^-$.
 Thus we  proved that
 \begin{eqnarray}\label{eqnpourtheo2}
2^{2m-1}\Gamma^{1,2m-1}(X,b)=p_{01}^*(\Delta_X^-)\cdot \ldots \cdot
p_{0,2m-1}^*\Delta_X^-\,\,{\rm in}\,\,{\rm CH}(X^{2m}).
\end{eqnarray}
Let $i'$ be the involution $(i,Id,\ldots,Id)$ acting on $X^{2m}$.
Observe that each cycle $p_{0j}^*\Delta_X^-$ is skew-invariant under
${i'}^*$. It follows from (\ref{eqnpourtheo2}) that
$p_{01}^*(\Delta_X^-)\cdot \ldots \cdot p_{0,2m-1}^*\Delta_X^-$ is
skew-invariant under ${i'}^*$, hence also under $i'_*={i'}^*$. But
as we have $p_{1,\ldots,2m-1}\circ i'=p_{1,\ldots,2m-1}$, we have

$$\Gamma^{2m-1}(X,b)=p_{1,\ldots,2m-1*}(\Gamma^{1,2m-1}(X,b))=p_{1,\ldots,2m-1*}\circ i'_*(\Gamma^{1,2m-1}(X,b))$$
$$= -p_{1,\ldots,2m-1*}(\Gamma^{1,2m-1}(X,b))=-\Gamma^{2m-1}(X,b),$$ so that
$\Gamma^{2m-1}(X,b)=0$ in ${\rm CH}_n(X^{2m-1})$.

\end{proof}
We now turn to the proof of Theorem \ref{theointro2}, (ii) and (iii) : in fact, the
result will take the following
more precise form:

\begin{theo}\label{theodm} Let $\pi:X\rightarrow Y$ be a finite morphism of
degree $d$. If $\Gamma^m(Y,a)=0$ in ${\rm CH}(Y^m)/R$ for some
adequate equivalence relation $R$, and $b=\frac{1}{d}p^*a$ satisfies
\begin{eqnarray} \label{eqimportant} b*b=\Delta_*(b)\,\,{\rm in}\,\,{\rm CH}_0(X\times X)/R,
\end{eqnarray}
where $\Delta:X\rightarrow X\times X$ is the diagonal inclusion map,
then $\Gamma^{d(m-1)+1}(X,b)=0$ in ${\rm CH}(X^{dm})/R$.

\end{theo}

Statement (ii) of Theorem \ref{theointro2} is the case where $R$ is rational equivalence
(that is $R=0$) and $b$ is the class of
a  point of $X$, as all points satisfy (\ref{eqimportant}) modulo rational equivalence.
Statement (iii) of Theorem \ref{theointro2} is the case where $R$ is algebraic equivalence.
Indeed, Theorem \ref{theodm} applies since the equality
$b*b=\Delta_*(b)\,\,{\rm in}\,\,{\rm CH}_0(X\times X)$ modulo algebraic equivalence
is satisfied by $0$-cycles of degree $1$ on a connected variety.

We first introduce some notation. Let as above
$\Delta_Y^a:=\Delta_Y-p_2^*a\in {\rm CH}_n(Y\times Y)$ and similarly
$\Delta_X^b:=\Delta_X-p_2^*b\in {\rm CH}_n(X\times X)$. In both
expressions, $p_2$ is the  projection from $Y\times Y$, resp.
$X\times X$ onto its second factor. The proof  of Theorem
\ref{theodm} will use  the following result (which will replace
formula (\ref{eqegacruc}) used previously when $d=2$):
\begin{prop} \label{legroupd}
The morphism $\pi:X\rightarrow Y$ and the $0$-cycle $b$ being as in
Theorem \ref{theodm}, there exist  cycles $\Gamma_i\in {\rm
CH}^{(d-1)n}(X^{d+1})$ such that
\begin{eqnarray}\label{eqformcrucpourdn}
\prod_{i=1}^dp_{0i}^*\Delta_X^{b}=\sum_ip_{0i}^*(\pi_2^*\Delta_Y^a)\cdot
p_{0,D\setminus\{i\}}^*\Gamma_i \,\,{\rm in}\,\, {\rm
CH}^{nd}(X^{d+1})/R,
\end{eqnarray}
where $D$ is the set $\{1,\ldots,d\}$ and as usual
$p_{0,D\setminus\{i\}}$ is the projection onto the product of the factors indexed
by the set $\{0\}\cup D\setminus\{i\}$.
\end{prop}
Before giving the proof, we will first  prove a similar statement of independent interest
for $\Delta_X$ and $\Delta_Y$, instead of $\Delta_X^b$ and $\Delta_Y^a$, as the proof is much
simpler to write and we will use similar but slightly more involved arguments to
prove Proposition \ref{legroupd}. Namely we have the following result:
\begin{prop} \label{peutservir} Let $\pi:X\rightarrow Y$ be a finite morphism of degree
$d$. There exist  cycles $\Gamma'_i\in {\rm CH}^{n(d-1)}(X^{d+1})$
such that
\begin{eqnarray}\label{eqformcrucpourdn2}
\prod_{i=1}^dp_{0i}^*\Delta_X=\sum_ip_{0i}^*(\pi_2^*\Delta_Y)\cdot
p_{0,D\setminus\{i\}}^*\Gamma'_i \,\,{\rm in}\,\, {\rm
CH}^{nd}(X^{d+1}).
\end{eqnarray}
\end{prop}
\begin{proof}
Indeed, let us denote by $E_{k}\subset {\rm CH}(X^{k+1})$ the ideal
generated by the elements $p_{0i}^*(\pi_2^*\Delta_Y)$, $i=1,\ldots,
k$. Next let \begin{eqnarray}\label{sigma118avril}
\Sigma_1:=\pi_2^{-1}(\Delta_Y)- \Delta_X\in {\rm CH}(X\times X).
\end{eqnarray} Note that, because
$\pi$ is finite of degree $d$, $\Sigma_1$ is the class of the
Zariski closure in $X\times X$ of the subvariety $\{(x,x_1)\in
X^0\times X^0,\,\pi(x_1)=\pi(x),\,x_1\not=x\}$ where
$X^0:=\pi^{-1}(Y^0)$ and $Y^0$ is the open set of $Y$ over which
$\pi$ is \'etale of degree $d$. The first projection
$pr_1:\Sigma_1\rightarrow X$ has degree $d-1$. Let us denote more
generally by $\Sigma_k\subset X^{k+1}$ the Zariski closure in
$X^{k+1}$ of the subvariety
\begin{eqnarray}
\label{defsigmak}\{(x,x_1,\ldots,x_k)\in
(X^0)^{k+1},\pi(x_i)=\pi(x),\,\,x_i\not=x_j\,\,{\rm
for}\,\,i\not=j,\,x_i\not=x\,\,{\rm for \,\,all}\,\,i\}.
\end{eqnarray}

The contents of formula (\ref{eqformcrucpourdn2}) is that
$\prod_{i=1}^dp_{0i}^*\Delta_X$ belongs to $E_{d}$. It is therefore a consequence of the following
statement:
\begin{claim} For any integer $k\geq 1$, one has
\begin{eqnarray}\label{eqsigma18avril}
\alpha_k\prod_{i=1}^kp_{0i}^*\Delta_X=\Sigma_k \,\,{\rm in
\,\,CH}(X^{k+1})/E_{k+1},
\end{eqnarray}
with $\alpha_k=(-1)^kk!$. In particular,
$\prod_{i=1}^dp_{0i}^*\Delta_X=0$ in ${\rm  CH}(X^{d+1})/E_{d}$.
\end{claim}
 The second statement follows from the first since $\Sigma_d$ is empty.
The first statement is proved by induction on $k$. For $k=1$, the
result  is (\ref{sigma118avril}). The induction step is immediate:
we have the following equalities in ${\rm CH}(X^{k+1})$:
\begin{eqnarray}\label{eqsigma18avril111}\prod_{i=1}^{k+1}p_{0i}^*\Delta_X=p_{0,\ldots,k}^*
(\prod_{i=1}^{k}p_{0i}^*\Delta_X)\cdot p_{0,k+1}^*\Delta_X\\
\nonumber =-(\prod_{i=1}^{k}p_{0i}^*\Delta_X)\cdot
p_{0,k+1}^*\Sigma_1\,\,{\rm mod}\,\,E_{k+1}\\
\nonumber =-\frac{1}{\alpha_k}p_{0,\ldots,k}^*(\Sigma_k)\cdot
p_{0,k+1}^*\Sigma_1\,\,{\rm mod}\,\,E_{k+1}. \end{eqnarray} On the
other hand, we observe that $\Sigma_{k+1}$ is obtained from
$p_{0,\ldots,k}^*(\Sigma_k)\cdot p_{0,k+1}^*\Sigma_1$ by removing in
the fibered product the components where $x_{k+1}$ equals one of the
$x_i$'s for $i=1,\ldots,k$. This gives rise to the following
identity:
\begin{eqnarray}\label{numero118avril}
p_{0,\ldots,k}^*(\Sigma_k)
\cdot p_{0,k+1}^*\Sigma_1=\Sigma_{k+1}+\sum_{i=1}^kp_{0,\ldots,k}^*(\Sigma_k)\cdot
p_{i,k+1}^*\Delta_X.
\end{eqnarray}
In the right hand side of (\ref{numero118avril}), we can replace
(using again the induction hypothesis) $\Sigma_k$ by
${\alpha_k}\prod_{j=1}^{k}p_{0j}^*\Delta_X$ mod $E_k$ and we also
observe that
\begin{eqnarray}\label{numero118avri1111l}\prod_{j=1}^{k}p_{0j}^*\Delta_X\cdot
p_{i,k+1}^*\Delta_X=\prod_{i=1}^{k+1}p_{0i}^*\Delta_X
\end{eqnarray}
 for any
$i=1,\ldots,k$. Hence we get, using (\ref{eqsigma18avril111}),
(\ref{numero118avril}) and (\ref{numero118avri1111l}),
$$\prod_{i=1}^{k+1}p_{0i}^*\Delta_X
=-\frac{1}{\alpha_k}\Sigma_{k+1}-k\prod_{i=1}^{k+1}p_{0i}^*\Delta_X.$$

This finally provides
$$ \alpha_{k+1}\prod_{i=1}^{k+1}p_{0i}^*\Delta_X=\Sigma_{k+1}$$
with $\alpha_{k+1}=-(k+1)\alpha_k$.
\end{proof}

\begin{proof}[Proof of Proposition \ref{legroupd}]
We   follow the above argument with $\Delta_X$, $\Delta_Y$ replaced
by $\Delta_X^b$ and $\Delta_Y^a$, in order to prove Lemma
\ref{le20avril} below. We use the following notation: we will work
with the $n$-cycle $\Sigma_k^b$ of $X^{k+1}$ obtained by replacing
formally in the definition (\ref{defsigmak}) of $\Sigma_k$ each
$x_i$ by $x_i-b$ and developing multilinearly. More rigorously,
$\Sigma_k$ admits morphisms $p,\,p_i:\Sigma_k\rightarrow X$,
obtained by restricting the projections $X^{k+1}\rightarrow X$
(where the factors are indexed by $\{0,\ldots,k\}$ and $p=p_0$). Let
$\Gamma_i\subset \Sigma_{k}\times X$ be the graphs of these
projections. Then we can obviously define $\Sigma_k\subset X^{k+1}$
as $(p,pr_{X^k})_*(\prod_{i=1}^kpr_{\Sigma_k,i}^*\Gamma_i)$, where
\begin{itemize}
\item $pr_{X^k}:\Sigma_k\times X^k\rightarrow X^k$ is the second
projection and $(p,pr_{X^k}):\Sigma_k\times X^k\rightarrow X^{k+1}$
is the obvious morphism.
\item $pr_{\Sigma_k,i}: \Sigma_k\times X^k\rightarrow \Sigma_k\times
X$ is the projection on the product of the first factor and the
$i$-th factor of $X^k$.
\end{itemize}
On the other hand, we also have in  $\Sigma_{k}\times X$  the graph
$\Sigma_k\times\{b\}$ of the constant morphism mapping to $b$ if $b$
is a point, or more generally the $n$-cycle $pr_X^*b$ if $b$ is any
$0$-cycle of degree $1$. We then define analogously $\Sigma_k^b$ as
follows:
\begin{eqnarray} \Sigma_k^b=(p,pr_{X^k})_*(\prod_{i=1}^kpr_{\Sigma_k,i}^*(\Gamma_i-pr_X^*
b))\,\,{\rm in}\,\,{\rm CH}(X^{k+1}).
\end{eqnarray}
 Developing the product above, we see that the formula for $\Sigma_k^b$ is
 of the form
\begin{eqnarray}\label{eqpourSigmakb}\Sigma_k^b=\sum_{I\subset
\{1,\ldots,k\}}(-1)^{k-i}\lambda_{k,i,d}p_{0,I}^*\Sigma_i\cdot
p_{J}^*b^{*j}\in {\rm CH}_n(X^{k+1}),
\end{eqnarray}
where in the formula above, $I\sqcup J=\{1,\ldots,k\}$, $i=|I|$, and
the $\lambda_{k,j,d}$ are combinatorial coefficients given by the
formula
\begin{eqnarray}\label{eqpourlambdaij}
\lambda_{k,i,d}=(d-i-1)(d-i-2)\ldots(d-k).
\end{eqnarray}
Indeed, the reason for (\ref{eqpourlambdaij})  is the fact that the
projection map
$$p_{0,I}:\Sigma_k\rightarrow \Sigma_i\subset X^{i+1}$$
has degree $(d-i-1)(d-i-2)\ldots(d-k)$.
 Note  in particular, that $\Sigma_k^b=0$ for $k\geq d$.
 Next we define
$E_{k,a,R}\subset {\rm CH}(X^{k+1})/R$ as the ideal generated by the
$p_{0,i}^*\Delta_Y^a$ for $i=1,\ldots,k$. Recall that
$\Gamma^{1,k}(X,b)=\prod_{i=1}^kp_{0i}^*\Delta_X^b$.
\begin{lemm} \label{le20avril}
The morphism $\pi:X\rightarrow Y$ and the $0$-cycle $b$ being as in Theorem \ref{theodm},
for any integer $k\geq 1$, one has
\begin{eqnarray}\label{eqsigma20avril}
\alpha_k\Gamma_{1,k}(X,b)=\Sigma_{k}^b \,\,{\rm in\,\,
CH}(X^{k+1})/E_{k,a,R},
\end{eqnarray}

\end{lemm}
\begin{proof}
We have by (\ref{eqpourSigmakb}), (\ref{eqpourlambdaij})
$$\Delta_X^b=\pi_2^*\Delta_Y^a-\Sigma_1^b,$$ which can be written as
$\Delta_X^b=-\Sigma_1^b$ mod $E_{1,a,R}$, proving the case $k=1$.
Assume the formula is proved for $k$. Then we have
\begin{eqnarray}\label{eqsigma20avril1}p_{0,1,\ldots,k}^*\Sigma_{k}^b\cdot p_{0,k+1}^*\Sigma_1^b=
-\alpha_k p_{0,1,\ldots,k}^*\Gamma^{1,k}(X,b)\cdot
p_{0,k+1}^*\Delta_X^b
\\ \nonumber
=-\alpha_k \Gamma^{1,k+1}(X,b)\,\,{\rm in}\,\,{\rm
CH}(X^{k+1})/E_{k,b,R}.
\end{eqnarray}
Next we claim that we have the following relation in ${\rm
CH}(X^{k+2})/R$:
\begin{eqnarray}\label{eqsigma21avril1}p_{0,1,\ldots,k}^*\Sigma_{k}^b\cdot p_{0,k+1}^*\Sigma_1^b=
\Sigma_{k+1}^b+\sum_{i=1}^{k}p_{0,\ldots,k}^*\Sigma_k^b\cdot p_{i,k+1}^*\Delta_X^b\\
\nonumber -\sum_{i=1}^{k}p_{0,\ldots,\hat{i},k+1}^*\Sigma_k^b\cdot
p_i^*b.
\end{eqnarray}
This relation uses in a crucial way the identity
\begin{eqnarray}\label{eqiddiagb26aout}
\Delta_*b=p_1^*b\cdot p_2^*b\,\,{\rm in}\,\,{\rm CH}_0(X\times X)/R.
\end{eqnarray}
 The beginning
$$p_{0,1,\ldots,k}^*\Sigma_{k}^b\cdot p_{0,k+1}^*\Sigma_1^b=
\Sigma_{k+1}^b+\sum_{i=1}^{k}p_{0,\ldots,k}^*\Sigma_k^b\cdot
p_{i,k+1}^*\Delta_X^b+...$$ of the formula (\ref{eqsigma21avril1})
is easily understood: it expresses the fact that in the left hand
side, we include all possible $x_{k+1}\not=x$, while in
$\Sigma_{k+1}^b$, we  have to take into account  the restriction
$x_{k+1}\not=x_i$ for $i=1,\ldots,k$. The last term in
(\ref{eqsigma21avril1}) is explained as follows. The intersection
with $p_{i,k+1}^*\Delta_X^b=p_{i,k+1}^*\Delta_X-p_{k+1}^*b$ produces
a term
$$\Delta_*(x_i-b)-(x_i-b,b)=(x_i,x_i)-\Delta_*b-(x_i,b)+p_i^*bp_{k+1}^*b$$
$$= (x_i,x_i)-(x_i,b)$$
on the product of the $i$th and $k+1$th factors. On the other hand,
we   had on the left in (\ref{eqsigma21avril1}) the  term
$$(x_i-b,x_i-b)=(x_i,x_i)-(x_i,b)-(b,x_i)+(b,b)$$ on the product of the $i$th and
$k+1$th factors, which is unwanted in the development of
$\Sigma_{k+1}^b$. Hence we also have to add on the right the extra
term $-(b,x_i-b)$ on the product of the $i$th and $k+1$th factors,
which is exactly the meaning of the term
$-p_{0,\ldots,\hat{i},k+1}^*\Sigma_k^b\cdot p_i^*b$.
 Thus the
claim is proved.

 Combined with
(\ref{eqsigma20avril1}) and the inductive assumption,
(\ref{eqsigma21avril1}) gives

\begin{eqnarray}\label{eqsigma20avril2}
-\alpha_k \Gamma^{1,k+1}(X,b)= \Sigma_{k+1}^b
\\ \nonumber
+\alpha_k(\sum_{i=1}^{k}p_{0,\ldots,k}^* \Gamma^{1,k}(X,b)\cdot
p_{i,k+1}^*\Delta_X^b-\sum_{i=1}^{k}p_{0,\ldots,\hat{i},k+1}^*
\Gamma^{1,k}(X,b)\cdot p_i^*b).
\end{eqnarray}
The equality above holds in ${\rm CH}(X^{k+2})/E_{k+1,a,R}$. Let us
now  prove that for any $i$,
 \begin{eqnarray}\label{eq27mai}p_{0,\ldots,k}^*\Gamma^{1,k}(X,b)\cdot
p_{i,k+1}^*\Delta_X^b-p_{0,\ldots,\hat{i},k+1}^*
\Gamma^{1,k}(X,b)\cdot p_i^*b=\Gamma^{1,k+1}(X,b)
\end{eqnarray}
in ${\rm CH}(X^{k+2})/R$.
 As
$$\Gamma^{1,k}(X,b)=\prod_{i=1}^kp_{0i}^*\Delta_X^b,\,\,\Gamma^{1,k+1}(X,b)=
\prod_{i=1}^{k+1}p_{0i}^*\Delta_X^b,$$ it clearly suffices to show
that the cycles $p_{01}^*\Delta_X^b\cdot
p_{12}^*\Delta_X^b-p_1^*b\cdot p_{02}^*\Delta_X^b$
 and $p_{01}^*\Delta_X^b\cdot p_{02}^*\Delta_X^b$
of $X^3$ are equal in ${\rm CH}(X^{3})/R$. We have
$$p_{01}^*\Delta_X^b\cdot p_{12}^*\Delta_X^b-p_1^*b\cdot p_{02}^*\Delta_X^b= (p_{01}^*\Delta_X-p_1^*b)\cdot
(p_{12}^*\Delta_X-p_2^*b)-p_1^*b\cdot p_{02}^*\Delta_X+p_1^*b\cdot p_2^*b$$
$$=p_{01}^*\Delta_X\cdot p_{12}^*\Delta_X-p_{01}^*\Delta_X\cdot p_2^*b-
p_1^*b\cdot p_{12}^*\Delta_X+p_1^*b\cdot p_2^*b-p_1^*b\cdot p_{02}^*\Delta_X+p_1^*b\cdot p_2^*b$$
$$=p_{01}^*\Delta_X\cdot p_{02}^*\Delta_X-p_{01}^*\Delta_X
\cdot p_2^*b-p_1^*b\cdot p_{02}^*\Delta_X+p_1^*b\cdot p_2^*b$$ in
${\rm CH}(X^3)/R$ because we assumed $p_1^*b\cdot
p_{12}^*\Delta_X=p_1^*b\cdot p_2^*b $ in ${\rm CH}(X^3)/R$ (cf.
(\ref{eqimportant})). On the other hand,
$$p_{01}^*\Delta_X^b\cdot p_{02}^*\Delta_X^b=(p_{01}^*\Delta_X-p_1^*b)\cdot (p_{02}^*\Delta_X-p_2^*b)$$
$$=p_{01}^*\Delta_X\cdot p_{02}^*\Delta_X-p_{01}^*\Delta_X\cdot p_2^*b-
p_1^*b\cdot p_{02}^*\Delta_X+p_1^*b\cdot p_2^*b.$$ Hence we proved
that both terms in (\ref{eq27mai}) are equal; using
(\ref{eqsigma20avril2}), we then get:
$$-\alpha_k \Gamma^{1,k+1}(X,b)=
\Sigma_{k+1,b}+\alpha_k(\sum_{i=1}^{k}\Gamma^{1,k+1}(X,b))
,$$
hence
$$-(k+1)\alpha_k \Gamma^{1,k+1}(X,b)=\Sigma_{k+1,b}\,\,{\rm in}\,\,{\rm CH}(X^{k+2})/R$$
and Lemma \ref{le20avril} is proved.

\end{proof}
Finally,  Lemma \ref{le20avril} for $k=d$  implies Proposition
\ref{legroupd} since $\Sigma_d^b=0$.

\end{proof}

\begin{proof}[Proof of Theorem \ref{theodm}]
By Lemma \ref{legroupd} applied to each set of $d$ indices
$\{1,\ldots,d\},\,\{d+1,\ldots,2d\}$, $\{(m-2)d+1,\ldots,(m-1)d\}$,
we can write $\prod_{i=1}^{d(m-1)}p_{0i}^*\Delta_X^{b}$ as a sum of
products of $m-1$ cycles, each of them being of the form
$p_{0i}^*(\pi_2^*\Delta_Y^a)\cdot \Gamma''$ for
 an adequate index
$i$ (one in each of the sets above). We now apply Proposition
\ref{propourdm1} to both $Y$ and $X$. Thus the assumption
$\Gamma^m(Y,a)=0$ implies that for some cycle $\Gamma_Y$ on $Y^m$,
$$\prod_{i=1}^{m-1}p_{0i}^*\Delta_Y^a=p_0^*a\cdot \Gamma_Y\,\,{\rm in}\,\,{\rm CH}(Y^m)/R.$$

Applying this relation to each product of $m-1$ factors
$\prod_{k=1}^{m-1}p_{0i_k}^*(\pi_2^*\Delta_Y^a)$ for adequate
indices $i_k$ appearing above, we conclude that
$$\Gamma^{1,d(m-1)}(X,b)=\prod_{i=1}^{d(m-1)}p_{0i}^*\Delta_X^{b}=p_{0}^*b\cdot\Gamma_X\,\,{\rm in}\,\,{\rm CH}(X^{d(m-1)+1})/R$$
for some cycle $\Gamma_X$ on $X^{d(m-1)+1}$. By Proposition
\ref{propourdm1}, and using the fact that $b$ satisfies property
(\ref{eqimportant}), (that is, condition (\ref{eqcyclea}) in
Proposition \ref{propourdm1}), we conclude that
$\Gamma^{d(m-1)+1}(X,b)=0$ in ${\rm CH}(X^{d(m-1)+1})/R$.

\end{proof}
\subsection{Case of curves\label{subsec}}
A special case of Theorem \ref{theointro2}, (iii) is the case where
$Y=\mathbb{P}^1$, so $X$ is a $d$-gonal curve. We then get the
vanishing $\Gamma^{d+1}(X,b)=0$ in ${\rm CH}_1(X^{d+1})/{\rm alg}$,
where $b$ is any point of $X$. Recall now the Beauville
decomposition of cycles on an abelian variety $A$ modulo rational or
algebraic equivalence:
$${\rm CH}_i(A)=\oplus_s{\rm CH}_i(A)_s,$$
with
$${\rm CH}_i(A)_s:=\{z\in  {\rm CH}_i(A),\,\,\mu_{k*}z=k^{2i+s}z\,\,{\rm for\,\,all}\,\,k\in
\mathbb{Z}^*\}$$ and similarly for Chow groups modulo algebraic
equivalence. Here $\mu_k:A\rightarrow A$ is the morphism $a\mapsto
ka$. Let now $X$ be a smooth genus $g$ projective curve and
$A:=J(X)$. $X$ has an embedding in $J(X)$ which is canonical up to
translation, hence determines a $1$-cycle $Z$ in $J(X)$, well
defined modulo algebraic equivalence. Thus we have a Beauville
decomposition $$Z=\sum_s Z_s\,\,{\rm in}\,\, {\rm CH}_1(A)/{\rm
alg}.$$ For nonvanishing results concerning the cycles $Z_s$ (when
$X$ is very general) and its decomposition, let us mention
\cite{fakh}, \cite{voiinfjac} (in the later paper, it is proved that
if $g\geq s^2/2$, then $Z_{s}\not=0$ modulo algebraic equivalence
for a very general curve $X$ of genus $g$).

  Let us show the
following:
\begin{prop} \label{propsansnom12mai} The vanishing of $\Gamma^{d+1}(X,b)$ in ${\rm
CH}_1(X^{d+1})/{\rm alg}$ is equivalent to the vanishing of
$Z_s,\,\forall s\geq{d-1}$, in ${\rm CH}_1( J(X))/{\rm alg}$.
\end{prop}
\begin{proof} It suffices to prove the result for $d\leq g-1$,
because we know by Theorem \ref{theointro2}, (iii) (see Remark
\ref{remanew25aout}) that $\Gamma^{d+1}(X,b)=0$ in ${\rm
CH}_1(X^{d+1})/{\rm alg}$ for some $d\leq g-1$. Assuming the
proposition  proved for $d\leq g-1$, this implies that $Z_s=0$ in
${\rm CH}_1( J(X))/{\rm alg}$ for all $ s\geq{g-1}$, and thus for
$d\geq g$, both vanishing statements are true.

 We thus assume $d\leq g-1$; note that the
cycle $\Gamma^{d+1}(X,b)$ is a $1$-cycle of $X^{d+1}$ which is
invariant under the action of the symmetric group $
\mathfrak{S}_{d+1}$, so that its vanishing in ${\rm
CH}_1(X^{d+1})/{\rm alg}$ is equivalent to the vanishing of its
image $\overline{\Gamma}^{d+1}(X,b)$ in ${\rm CH}_1(X^{(d+1)})/{\rm
alg}$. We now consider the inclusion
$$b_{g-d-1}:X^{(d+1)}\rightarrow X^{(g)}$$
$$z\mapsto z+(g-d-1)b$$
and claim that $\overline{\Gamma}^{d+1}(X,b)=0$ in ${\rm
CH}_1(X^{(d+1)})/{\rm alg}$ if and only if $b_{g-d-1
\,*}(\overline{\Gamma}^{d+1}(X,b))=0$ in ${\rm CH}_1(X^{(g)})/{\rm
alg}$. Indeed, there is an incidence correspondence
$$\Sigma\subset X^{(d+1)}\times
X^{(g)},\,\Sigma=\{(z,z'),\,z'=z+z''\,\,{\rm for\,\,some}\,\,z''\in
X^{(g-d-1)}\}.$$ It is not hard to see that, due to its special
form, the cycle $\overline{\Gamma}^{d+1}(X,b)$ satisfies
$$\Sigma^*(b_{g-d-1\,*}(\overline{\Gamma}^{d+1}(X,b)))=\overline{\Gamma}^{d+1}(X,b),$$
which proves the claim.

 The next step is to observe that the Griffiths group
of $1$-cycles homologous to $0$ modulo algebraic equivalence is a
birational invariant. This is elementary to show using resolution of
indeterminacies of birational maps, as it is invariant under blow-up
and is functorial under pushforward and pullbacks under generically
finite morphisms. As $X^{(g)}$ is birational to $J(X)$ via the Abel
map, we conclude that $\overline{\Gamma}^{d+1}(X,b)=0$ in ${\rm
CH}_1(X^{(d+1)})/{\rm alg}$ if and only if its image $W$ in $J(X)$
under the Abel map vanishes in ${\rm CH}_1(J(X))/{\rm alg}$.

Finally, we observe that a cycle appearing in the formula
(\ref{eqformgammam}) for $\Gamma^{d+1}(X,b)$, which is up to
permutation of the form
$$\{(x,\ldots,x,b,\ldots,b),\,x\in X\},$$
where $x$ appears $k$ times and $b$ appears $d+1-k$ times, maps
under the Abel map to a $1$-cycle of $J(X)$ algebraically equivalent
to $\mu_{k*}(Z)$. The vanishing of $W$ in ${\rm CH}_1(J(X))/{\rm
alg}$ thus gives
\begin{eqnarray}
\label{eqnvaniW1}\sum_{k=1}^{d+1}(-1)^{d+1-k}\binom{d+1}{k}\mu_{k*}Z=0\,\,{\rm
in }\,\,{\rm CH}_1(J(X))/{\rm alg}.
\end{eqnarray}
 Writing the Beauville decomposition
$$Z=\sum_s Z_s,$$
the vanishing of $W$ in ${\rm CH}_1(J(X))/{\rm alg}$ is equivalent
to

\begin{eqnarray}
\label{eqnvaniW2}
\sum_{k=1}^{d+1}(-1)^{d+1-k}\binom{d+1}{k}k^{2+s}Z_s=0,\,{\rm
in}\,\,{\rm CH}_1(J(X))/{\rm alg}
\end{eqnarray}
for any $s$.

We now have the following easy lemma:
\begin{lemm} We have
$\sum_{k=1}^{d+1}(-1)^{d+1-k}\binom{d+1}{k}k^{2+s}=0$ for $s\leq
d-2$, and
$$\sum_{k=1}^{d+1}(-1)^{d+1-k}\binom{d+1}{k}k^{2+s}\not=0$$ for
$s\geq d-1$.
\end{lemm}
This shows that the vanishing (\ref{eqnvaniW2}) is equivalent to the
vanishing of $Z_s$ for $s\geq d-1$.
\end{proof}
\begin{rema}{\rm Proposition \ref{propsansnom12mai} is also proved
in \cite{moonenyin}, where it is used to deduce the vanishing
$\Gamma^{g+2}(X,a)=0$ of Remark \ref{remaref}, for any point $a\in
X$, from the main result of Colombo and van Geemen
\cite{colombovgeemen}.}
\end{rema}
\section{Hyper-K\"ahler manifolds\label{sec3}}
\subsection{Proof of Theorem \ref{theoHKintro}}
We prove in this section the following theorem (cf. Theorem
\ref{theoHKintro} of the introduction):
\begin{theo} \label{theoHK} Let $S$ be a $K3$ surface, and let $X=S^{[n]}$. Then
\begin{eqnarray}
\Gamma^{2n+1}(X,o_X)=0 \,\,{\rm in \,\, CH}_{2n}(X^{2n+1}),
\label{eqformk3prouvefin}
\end{eqnarray}
where $o_X$ is the canonical $0$-cycle on $X$ constructed from the canonical $0$-cycle of
$S$.
\end{theo}
Here the  cycle $o_S$ appears in the
 following theorem from \cite{beauvoi}  providing a list of relations
 which hold in the Chow ring of a self-product
of a $K3$ surface.
\begin{theo}\label{theorappelrelbeauvoi} Let $S$ be a smooth projective $K3$ surface. Then
there is a degree $1$ zero-cycle $o_S\in {\rm CH}_0(S)$ satisfying
the following equalities (which are all polynomial relations in
${\rm CH}(S^k)$ for adequate $k$, between the cycles
$p_i^*o_S,\,p_j^*L,\,p_{st}^*\Delta_S$):
\begin{enumerate}
\item $L^2-{\rm deg}\,(L^2)\,o_S=0\,\,{\rm in}\,\,{\rm CH}_0(S),\,{\rm for\,\,any\,\,}L\in Pic\,S$.
 \item $\Delta_S.p_1^*L-L\times o_S-o_S\times L=0$  in ${\rm CH}_1(S\times S)$
  for any $L\in {\rm Pic}\,S$,
where $p_1$  is the first projection from $S\times S$ to $S$, and
$L\times o_S=p_1^*L\cdot p_2^*o_S$.
\item \label{5} $\Gamma^3(S,o_S)=0$ in ${\rm
CH}_2(S\times S\times S)$.  (Using formula (\ref{eqformgammam}) and
the identity
$\Delta_3=p_{12}^*\Delta_S\cdot p_{13}^*\Delta_S$,
we can also see Property \ref{5} as a polynomial relation in ${\rm
CH}(S^3)$ involving the classes $p_{ij}^*\Delta_S$ and
$p_k^*(o_S)$.)
\item \label{itemsansnom} $\Delta_S^2=24\,p_1^*o_S\cdot p_2^*o_S$ in ${\rm CH}_0(S\times S)$.
\item \label{item3} $\Delta_S.p_1^*o_S-p_1^*o_S\cdot p_2^*o_S=0$ in ${\rm CH}_0(S\times S)$.
\end{enumerate}
\end{theo}
Note that property \ref{item3} is (\ref{eqimportant}) and is easily
satisfied because $o_S$ is the class of a point in $S$. Property
\ref{itemsansnom} is a consequence of Property \ref{5} which implies
$c_2(S)=24 \, o_S$ in ${\rm CH}_0(S)$, and Property \ref{item3}.
\begin{rema}\label{rematrivialrel} {\rm The above relations are the nontrivial relations
involving $p_i^*(o_S),\,p_j^*L,\,L\in{\rm CH}^1(S) $ and the $p_{kl}^*\Delta_S$
and with the property that in at least one monomial, an index is repeated. To make a complete
list of such relations, one should add the ``trivial relations'', which hold on any surface, namely
\begin{enumerate}
\item $o_X\cdot L=0$ in ${\rm CH}(S)$, $L\in {\rm CH}^1(S)$,
\item $o_X\cdot o_X=0$ in ${\rm CH}(S)$,
\item $p_{12}^*\Delta_S\cdot p_{23}^*\Delta_S=p_{13}^*\Delta_S\cdot p_{23}^*\Delta_S$ in ${\rm CH}(S\times S\times S)$.
\end{enumerate}}
\end{rema}

As in
\cite{voisinpamq}, the ingredients of the proof of Theorem \ref{theoHK} are
   1) the results of de Cataldo-Migliorini
\cite{decami}, which will allow, thanks to Proposition
\ref{proppourtheofin}, to translate the problem into
computations
in  ordinary
self-products $S^N,\, N\leq (2n+1)n$, of a $K3$ surface; 2) the
relations listed in Theorem \ref{theorappelrelbeauvoi}; 3)  the
recent result of Yin \cite{yin}. The latter says basically that for
a regular surface $S$, the {\it cohomological} polynomial relations on
$S^N$ between the diagonal classes and the pull-back under the
various projections of the class of a point are generated by the
relations listed above (or rather, their cohomological counterpart)
and the Kimura relation (cf. \cite{kimura},
\cite[3.2.3]{voisinbook}) which holds when the motive of $S$ is
finite dimensional. A key point of the proof will be  thus
 the fact that the Kimura relation is not needed to express the pull-back to
 $S^N$ of the vanishing relation
 $[\Gamma^{2n+1}(X,o_X)]=0$.

We first recall some notation related to $S^n$ and $S^{[n]}$, for
any smooth surface $S$. Let $\mu=\{A_1,\ldots,A_l\},\,l=:l(\mu)$ be
a partition of $\{1,\ldots,n\}$, where all the $A_i$'s are nonempty.
Let $S^\mu\cong S^{l(\mu)}\subset S^n$ be the set $$\{(s_1,\ldots,
s_n),\,s_i=s_j\,\,{\rm if}\,\,i,\,j\in A_k\,\,{\rm
for\,\,some}\,\,k\}.$$ The image $\overline{S^{(\mu)}}$ of $S^\mu$
in $S^{(n)}$ is a stratum of $S^{(n)}$. It is not normal in general, but  its
normalization $S^{(\mu)}$ is the quotient of $S^\mu$ by the subgroup
$\mathfrak{S}_\mu$ of $\mathfrak{S}_n$ preserving $S^\mu$, that is
acting on $\{1,\ldots,n\}$ by permuting the $A_i$'s with the same
cardinality.
 Let $c: S^{[n]}\rightarrow S^{(n)}$ be
the Hilbert-Chow morphism and let $E_\mu:=S^\mu\times_{S^{(n)}}
S^{[n]}\subset S^\mu\times S^{[n]}$. It is known that $E_\mu$ is
irreducible of dimension $n+l(\mu)$. We see $E_\mu$ as a
correspondence between $S^\mu$ and $S^{[n]}$.
\begin{theo} (de Cataldo-Migliorini \cite{decami}) \label{dedecami} The collection  $(E_\mu)_\mu$ of
correspondences identifies the motive of $S^{[n]}$ to a submotive of
the disjoint union  $\sqcup_\mu S^\mu$. More precisely, for some
combinatorial coefficients $\lambda_\mu$,
$$\Delta_X=\sum_\mu\lambda_\mu(E_{\mu},E_\mu)_*(\Delta_{S^\mu})\,\,{\rm
in}\,\,{\rm CH}_{2n}(X\times X).$$
\end{theo}
The result above implies in particular: \begin{coro}\label{coroinj}
Let $X:=S^{[n]}$. For any integer $k$, the map
$$\oplus_{(\mu_1,\ldots,\mu_k)}(E_{\mu_1},\ldots,E_{\mu_k})^*:{\rm CH}^*(X^k)\rightarrow
\oplus_{(\mu_1,\ldots,\mu_k)}{\rm CH}^*(S^{\mu_1}\times\ldots\times
S^{\mu_k})$$ is injective.
\end{coro}
We now have the following result: Let $n$ and $k$ be fixed.
  Let us denote
by $\Delta_k\subset X^k$ the small diagonal of $X^k$, where
$X:=S^{[n]}$, for
 a smooth projective surface $S$.
\begin{prop} \label{proppourtheofin}  For any $k$-uple  $(\mu_1,\ldots,\mu_k)$ of partitions of
$\{1,\ldots,n\}$, there exists a universal (i.e. independent of $S$) polynomial $P_{\mu_\cdot}$ (in many
variables) with the following property: For any smooth quasi-projective surface
$S$,
$$(E_{\mu_1},\dots,E_{\mu_k})^*(\Delta_k)=
P_{\mu_\cdot}(pr_i^*c_2(S),\,pr_j^*(K_S),pr_{st}^*(\Delta_S))\,\,{\rm
in}\,\, {\rm CH}(S^{\mu_1}\times\ldots\times S^{\mu_k}) ,$$ where
the $pr_i$'s are the projections from $\prod_iS^{\mu_i}\cong S^N$ to
its factors (isomorphic to $S$), and the $pr_{st}$'s are the
projections from $\prod_iS^{\mu_i}$ to the products of two of its
factors (isomorphic to $S\times S$).
\end{prop}
\begin{proof}
Proposition \ref{proppourtheofin} is a particular case of Theorem
\ref{theopourtheofin} whose proof will be sketched in Subsection
\ref{secuniv} and will be completed in \cite{voisinuniversalcycles}, because
the cycles $(E_{\mu_1},\dots,E_{\mu_k})^*(\Delta_k) \in {\rm CH}(S^N)
$ are clearly universally defined cycles in the sense of Definition \ref{defiunivdef}. Indeed, for any family $\mathcal{S}\rightarrow B$ of smooth quasi-projective surfaces,
 we can construct the smooth family of relative Hilbert schemes
$\mathcal{X}:=\mathcal{S}^{[n/B]}$ and its relative small diagonals
$$\Delta_{k/B}(\mathcal{X})\subset \mathcal{X}^{k/B}.$$
Then we have the relative correspondences $E_{\mu_i}\subset \mathcal{S}^{\mu/B}\times_B\mathcal{X}$,
which are proper over the first summand,
and we have thus the relative cycle
$$E_{\mu_\cdot}^*(\Delta_{k/B}(\mathcal{X}))\in{\rm CH}(\mathcal{S}^{[N/B]}),\,\,N=l(\mu_1)+\ldots+l(\mu_k),$$
 satisfying the functoriality properties stated in Definition
\ref{defiunivdef}, because the morphisms $E_{\mu_i}\rightarrow B$ are flat.
\end{proof}
\begin{rema} {\rm One may have the feeling that the canonical class
is not necessary in Proposition \ref{proppourtheofin}, as  set theoretically  one wants the set of
$(s_1,\ldots,s_k)\in S^{\mu_1}\times \ldots\times S^{\mu_k}$ such
that there is a subscheme $x\in S^{[n]}$ whose associated cycle is
$s_i$ (or rather its image in $S^{(n)}$) and this does not seem to
involve the intrinsic geometry of $S$, except for the
self-intersection of the diagonal, thus only $c_2$. In fact, due to
excess formulas, the canonical class actually appears, as the
simplest example shows: Let $X$ be $S^{[2]}$, $k=3$, and
$\mu_1=\mu_2=\mu_3$ be the partition of $\{1,2\}$ consisting of a
single set with $2$ elements. Then $E_{\mu_1}=E_{\mu_2}=E_{\mu_3}=E$
is the exceptional divisor of $S^{[2]}$ and we have
$$(E_{\mu_1},E_{\mu_2},
E_{\mu_3})^*(\Delta_3)=\Delta_*(p_*(E^2_{\mid E})),$$ where
$\Delta:S\rightarrow S^3$ is the diagonal inclusion, and
$p:E\rightarrow S$ is the natural map. But $p_*(E^2_{\mid E})\in
{\rm CH}^1(S)$ is a nonzero multiple of the canonical class of $S$.}
\end{rema}
\begin{rema}{\rm We proved in  \cite{voisinpamq}
 similar statement where instead of the small
diagonal, arbitrary polynomials in the Chern classes of the
tautological sheaf on $X^{[n]}$ and the Chern classes of the ideal
sheaf of the incidence correspondence in $S^{[n-1]}\times S$) are
considered;  the same kind of arguments used in {\it loc. cit.}, which are in fact borrowed from
\cite{EGL}, can  be applied to prove Proposition \ref{proppourtheofin}, but the proofs
are very intricated and lengthy and in fact
all these results can  also be obtained  as Proposition  \ref{proppourtheofin},
as a consequence
of Theorem
\ref{theopourtheofin}.
}
\end{rema}

We now show how  Theorem \ref{theoHK} follows from Proposition
\ref{proppourtheofin}.
\begin{proof}[Proof of Theorem \ref{theoHK}]   We have to prove the vanishing of
$\Gamma^{2n+1}(X,o_X)$, where
 $S$ is a smooth projective $K3$
surface and $X=S^{[n]}$. By Corollary \ref{coroinj}, it suffices to show that for
any $2n+1$-uple $(\mu_1,\ldots,\mu_{2n+1})$ of partitions of
$\{1,\ldots,n\}$, we have
\begin{eqnarray}\label{vanemustar}
(E_{\mu_1},\dots,E_{\mu_{2n+1}})^*(\Gamma^{2n+1}(X,o_X))=0\end{eqnarray}
in ${\rm CH}(S^{\mu_1}\times\ldots\times S^{\mu_{2n+1}})$. As
$\Gamma^{2n+1}(X,o_X)$ is a combination of cycles which up to
permutation of factors are of the form $\Delta_k\times o_X^{2n+1-k}$
and $E_\mu^*o_X=0$ if $\mu\not=\{\{1\},\ldots,\{n\}\}$, and is equal
to $n!(o_S,\ldots,o_S)$ if $\mu=\{\{1\},\ldots,\{n\}\}$, it follows
from Proposition \ref{proppourtheofin} that there exists a
polynomial $Q_{\mu_\cdot}$ (in many variables) with the following
property: For any smooth projective  surface $S$, and any point
$o_S\in S$,
\begin{eqnarray}
\label{eqEmustarQ}(E_{\mu_1},\dots,E_{\mu_{2n+1}})^*(\Gamma^{2n+1}(X,o_X))=
Q_{\mu_\cdot}(pr_i^*c_2(S),\,pr_j^*(K_S),\,pr_l^*o_S,\,pr_{st}^*(\Delta_S))
\end{eqnarray}
 in ${\rm CH}(S^{\mu_1}\times \ldots\times S^{\mu_{2n+1}})$.
 We know by \cite[Proposition 1.3]{ogrady} (see also Theorem \ref{theoog}, (i))
  that for any regular surface
 $S$, and any point $o_S\in S$,
$\Gamma^{2n+1}(X,o_X)$ is cohomologous to $0$, where $o_X$ is any
point of $X=S^{[n]}$ over $no_S\in S^{(n)}$. It follows that for
each $2n+1$-uple $(\mu_1,\ldots,\mu_{2n+1})$ as above, the cycle
$$(E_{\mu_1},\dots,E_{\mu_{2n+1}})^*(\Gamma^{2n+1}(X,o_X))$$ is
cohomologous to $0$ in $S^{\mu_1}\times\ldots\times S^{\mu_{2n+1}}$.
Hence the polynomial $Q_{\mu_\cdot}$ has the property that for a
regular surface $S$,
\begin{eqnarray}\label{eqvancohpol}
Q_{\mu_\cdot}(pr_i^*[c_2(S)],\,pr_j^*([K_S]),\,pr_l^*[o_S],\,pr_{st}^*([\Delta_S]))=0\,\,{\rm
in}\,\, H^*(S^{\mu_1}\times\ldots\times S^{\mu_{2n+1}},\mathbb{Q}).
\end{eqnarray}
Here the brackets denote the cohomology class of the corresponding
cycles. In this equation, we can of course replace $[c_2(S)]$
by $\chi_{top}(S)[o_S]$, with $\chi_{top}(S)$
determined by the polynomial relation (this is relation \ref{itemsansnom} in Theorem \ref{theorappelrelbeauvoi})
$[\Delta_S]^2=\chi_{top}(S)pr_1^*[o_S]\cup pr_2^*[o_S]$ in $H^4(S\times S,\mathbb{Q})$. We now follow \cite{voisinpamq} (see also \cite{yin}): The
cohomological version of the equations given in Theorem
\ref{theorappelrelbeauvoi} with $L=K_S$ holds on any smooth
projective surface with $b_1=0$, and if the canonical class
satisfies $[K_S]=0$ or $[K_S]^2\not=0$, one can  reduce  modulo
these relations  any polynomial expression in the variables
$$pr_i^*[pt],\,pr_j^*[K_S],\,pr_{st}^*[\Delta_S]$$
to a linear combination of monomials in the variables
$pr_i^*[pt],\,pr_j^*[K_S],\,pr_{st}^*[\Delta_S]^0$,
with the property that no index appears twice in the monomial. Here
the class $[\Delta_S]^0$ is the class
$$[\Delta_S]-pr_1^*[pt]-pr_2^*[pt]-\lambda pr_1^*[K_S]\cup pr_2^*[K_S],$$
where the coefficient $\lambda$, when $K_S\not=0$, is determined by
the relation $\lambda [K_S]^2=1$ (the class $[\Delta_S]^0\in
H^4(S\times S,\mathbb{Q})$) is the projector onto
$H^2(S,\mathbb{Q})^{{\perp}[K_S]}$).
Now, it is clear by K\"unneth decomposition that if
a linear combination of such monomials vanishes in $H^*(S^N,\mathbb{Q})=H^*(S,\mathbb{Q})^{\otimes N}$, then for fixed distinct indices $i_1,\ldots,i_m,\,j_1,\ldots,j_p,\,k_1,\ldots,k_q$, the
sum of such monomials of the form
 $$pr_{i_1}^*[pt]\cdot\ldots \cdot pr_{i_m}^* [pt]\cdot pr_{j_1}^*[K_S]\cdot\ldots\cdot
 pr_{j_p}^*[K_S]\cdot pr_{k_1}^*1_S\cdot\ldots\cdot pr_{k_q}^*1_S\cdot \prod_{s_1,t_1,\ldots,s_l,t_l}pr_{s_i,t_i}^*[\Delta_S]^0$$
  where  the indices $s_i,\,t_j$ exhaust the remaining indices, are all distinct and
  are different from
  the $i_s,\,j_s,\,k_s$,  has to be $0$.
This way, we reduced the problem to linear combinations of monomials of the form
 \begin{eqnarray}
 \label{eqmonomial}
 pr_{s_1t_1}^*[\Delta_S]^0\cdot\ldots\cdot
pr_{s_lt_l}^*[\Delta_S]^0 \end{eqnarray}
on $S^{2l}$,
 where no index is repeated.
We now have the following result due to Yin \cite{yin}: The ``Kimura
relation'' is a relation between monomials of the above type. It
says that, for $M={\rm dim}\,H^2(S,\mathbb{Q})^{{\perp}[K_S]}$, the
cohomology class of the projector onto
$\bigwedge^{M+1}H^2(S,\mathbb{Q})^{{\perp}[K_S]}\subset
H^{2M+2}(S^{M+1},\mathbb{Q})$ is $0$, which is obvious since
$\bigwedge^{M+1}H^2(S,\mathbb{Q})^{{\perp}[K_S]}=0$.
 The
class of this projector is the class \begin{eqnarray}
\label{eqkimura} \sum_{\sigma\in
\mathfrak{S}_{M+1}}\epsilon(\sigma)\prod_{i=1}^{M+1}pr_{i,M+1+\sigma(i)}^*[\Delta_S]^0\in
H^{4M+4}(S^{2M+2},\mathbb{Q}).
\end{eqnarray}
and the Kimura relation is thus the vanishing of (\ref{eqkimura}).
\begin{theo} (Yin, \cite{yin})\label{theoyin} For any integer $m$,
the relations in $H^*(S^m,\mathbb{Q})$ between the monomials
(\ref{eqmonomial}) with no repeated indices are generated by the
pull-back to $S^m$ of the Kimura relation via a projection (and a
permutation) $S^m\rightarrow S^{2N+2}$.

\end{theo}
We deduce the following
\begin{coro}\label{corodelafinfin} The polynomial $Q_{\mu_\cdot}$ belongs to the ideal
generated by the trivial relations (see Remark \ref{rematrivialrel}), the relation $c_2(S)=\chi_{top}(S) o_S$ (where we
recover $\chi_{top}(S)$ as the self-intersection of $\Delta_S$) and
the relations listed in Theorem \ref{theorappelrelbeauvoi} with
$L=K_S$.
\end{coro}
\begin{proof} Indeed, choose for $S$ a smooth projective surface
with $b_1(S)=0$ and $b_2(S)> \frac{n(2n+1)}{2}$. Then by Theorem
\ref{theoyin}, there are no linear relations between the  monomials
(\ref{eqmonomial}) with no repeated index if  $s\leq (2n+1)n$. On
the other hand, we have the vanishing of the cohomology class $$
Q_{\mu_\cdot}(pr_i^*[c_2(S)],\,pr_j^*[K_S],\,pr_l^*[o_S],\,pr_{st}^*[\Delta_S])
\in H^*(S^N,\mathbb{Q}),$$ where $N=\sum_il(\mu_i)\leq (2n+1)n$.
It
then follows from the above reduction that the polynomial
$Q_{\mu_\cdot}$, where one substitutes $\chi_{top}(S)[o_S]$ to
$[c_2(S)]$, belongs to the ideal generated by the cohomological
version of the relations given in Theorem
\ref{theorappelrelbeauvoi}, with $L=K_S$.
\end{proof}
The proof of Theorem \ref{theoHK} is now finished. Indeed, $S$ being
now a $K3$ surface, we know by Theorem \ref{theorappelrelbeauvoi}
that the relation $\chi_{top}(S)o_S=c_2(S)$ holds in $CH_0(S)$ and
that the relations listed in Theorem \ref{theorappelrelbeauvoi} hold
in ${\rm CH}(S^k)$ for adequate $k$. As the polynomial
$Q_{\mu_\cdot}$, where one substitutes $\chi_{top}(S)o_S$ to
$c_2(S)$, belongs to the ideal generated by the relations given in
Theorem \ref{theorappelrelbeauvoi} and the trivial relations, we conclude that
$Q_{\mu_\cdot}=0$ in ${\rm CH}(S^N)$. By (\ref{eqEmustarQ}), we proved the vanishing
(\ref{vanemustar})
$$(E_{\mu_1},\dots,E_{\mu_{2n+1}})^*(\Gamma^{2n+1}(X,o_X))=0\,\,{\rm
in}\,\,{\rm CH}(S^N),$$ which
concludes the proof.
\end{proof}
\subsection{Universally defined cycles \label{secuniv} }

This subsection is devoted to introducing the notion of ``universally defined
 cycles'' and to sketching the proof of a quite general statement which  will be fully proved in \cite{voisinuniversalcycles}. It concerns ``universally defined" cycles on self-products of surfaces.
We first explain the meaning of this expression. In the following, we work with Chow groups
with integral coefficients, and we will write
${\rm CH}(X)_\mathbb{Q}$ for cycles with $\mathbb{Q}$-coefficients.
\begin{Defi}\label{defiunivdef} Let $n,\,N$ be integers. A universally defined cycle on the $N$-th power of
smooth complex algebraic varieties $X$ of a given dimension $n$ consists in the following data:   for
each smooth family of $n$-dimensional algebraic varieties  $\mathcal{X}\rightarrow B$, where
$B$ is smooth quasiprojective,  a cycle $z_\mathcal{X}\in {\rm CH}(\mathcal{X}^{N/B})$ is given, satisfying
the following conditions:

(i)
If $r:B'\rightarrow B$ is a morphism, with induced morphism
$$R_N:(\mathcal{X}')^{N/B'}\rightarrow
\mathcal{X}^{N/B},\,\mathcal{X}':=\mathcal{X}\times_BB',$$ then
$$ z_{\mathcal{X}'}=R_N^*z_{\mathcal{X}}\,\,{\rm in}\,\,{\rm CH}((\mathcal{X}')^{N/B'}).$$

(ii) If $\mathcal{X}\rightarrow B$ is a family as above and
$\mathcal{Y}\subset \mathcal{X}$ is a Zariski open set, then
$$z_\mathcal{Y}=z_{\mathcal{X}\mid \mathcal{Y}^{N/B}}\,\,{\rm in}\,\,{\rm CH}(\mathcal{Y}^{N/B}).$$
\end{Defi}
\begin{theo}\label{theopourtheofin}  For any universally defined cycle $z$ on $N$-th powers of surfaces,
there exists a polynomial $P$ with rational coefficients, depending only on $z$, such that
for any smooth algebraic surface
$S$ defined over $\mathbb{C}$, $$z_S=P(pr_i^*c_1(S),pr_j^*c_2(S),pr_{rs}^*\Delta_S)\,\,{\rm in}\,\,{\rm CH}(S^N)_\mathbb{Q}.$$

\end{theo}
\begin{rema}{\rm One could introduce as well
universally defined cycles with $\mathbb{Q}$-coefficients, by replacing everywhere in the definition
above ${\rm CH}$ by ${\rm CH}_\mathbb{Q}$.
It is possible that the conclusion holds as well for universally defined cycles
with $\mathbb{Q}$-coefficients, but our present proof uses the integral structure.
}
\end{rema}
We will give some hints on the proof, with a complete proof only in the case $N=1$
(Proposition \ref{projusteS}) and the construction of the desired polynomials
(Corollary \ref{corpourthefinsd} and Proposition \ref{coroonnesaitjamais}). We refer
to \cite{voisinuniversalcycles} for a full treatment.
 Let us
 first show  how to produce such  polynomials.
Let $G:=G(2,5)$ be the Grassmannian of $2$-dimensional vector subspaces
in $\mathbb{C}^5$. Any smooth complex projective surface can be embedded in $G$,
for example by choosing $5$ general sections of a very ample vector bundle $E$ on $S$.
Let $\mathcal{O}_G(1)$ be the Pl\"ucker line bundle on $G$, and let $c\in {\rm CH}^2(G)$ be the second
Chern class of the tautological rank $2$ vector bundle on $G$.
We choose an integer $d$, and consider the universal family
$\mathcal{S}_d\rightarrow B$ of smooth surfaces in $G$ which are
complete intersections of $4$ members of $|\mathcal{O}_G(d)|$.
The smooth variety $B$ is thus
the  vector space $H^0(G,\mathcal{O}_G(d))^4$ and
$$\mathcal{S}_d\subset \mathcal{S}_{d,univ}$$
is the Zariski open set consisting of points where
$\mathcal{S}_{d,univ}\rightarrow B$ is smooth.
Here
$$\mathcal{S}_{d,univ}:=\{(b,x)\in B\times G,\,b=(f_{1,b},\ldots,f_{4,b}),\,f_{i,b}(x)=0\,\,\forall i\}.$$
There is an obvious morphism
$$f:\mathcal{S}_d\rightarrow G$$
given by the restriction to  $\mathcal{S}_d$ of the  second projection
$\mathcal{S}_{d,univ}\rightarrow G$,
which induces for any $N\geq 1$ the morphism
$$f_N: \mathcal{S}_d^{N/U}\rightarrow G^N$$
with induced pull-back morphism
$f_N^*:{\rm CH}(G^N)\rightarrow
{\rm CH}(\mathcal{S}_d^{N/U})$.
We now use the following result, which is one of the main ingredients in the proof of Theorem
\ref{theopourtheofin}:
\begin{prop} \label{leunivsurfd} For any integer $N>0$ and
sufficiently large $d$,
${\rm CH}(\mathcal{S}_d^{N/U})$ is generated as a ${\rm CH}(G^N)$-module by
 the relative partial diagonals
$\Delta_{I/U}(\mathcal{S}_d)$.
\end{prop}
Here $I$ denotes as usual a partition of $ \{1,\ldots,N\}$, determining a partial diagonal.
\begin{proof}[Proof of Proposition \ref{leunivsurfd}] By the localization exact sequence, it suffices
to prove the result with
$\mathcal{S}_d$ replaced by $\mathcal{S}_{d,univ}$.
Next consider the natural morphism
$$
f_N:\mathcal{S}_{d,univ}^{N/B}\rightarrow G^d.$$
The fiber of $f_N$ over a $N$-uple $(x_1,\ldots x_N)$ consists of the set of
$4$-uples $(\sigma_1,\ldots,\sigma_4)\in H^0(G,\mathcal{O}_G(d))^4$ having the property
that the $\sigma_i$'s  vanish at all points $x_i$.
As $d $  is large compared to $N$, any $k$ distinct points of $G$ with $k\leq N$ impose independent
conditions to $H^0(G,\mathcal{O}_G(d))$, and thus, denoting by $G^N_k$ the locally closed
subvariety  of $G^N$ consisting of $N$-uples with exactly $k$-distinct points,
which is the disjoint union of the diagonals $\Delta_I(G)$ with $l(I)=k$ (or rather of
the $\Delta_I^0(G):=\Delta_I(G)\setminus \cup_{J,l(J)<k}\Delta_J(G)$), we find that
$f_N^{-1}(G^N_k)$ is a Zariski open set in a vector bundle over $G^N_k$. It follows from the
localization exact sequence and $A^1$-invariance that ${\rm CH}( G^N_k)\stackrel{f_N^*}
{\rightarrow }
{\rm CH}(f_N^{-1}(G^N_k))$ is surjective.
Writing $G^N$ as the disjoint union of the $\Delta_I^0(G)$,
we conclude from the above and the localization exact sequence that
$$\oplus_I{\rm CH}(\Delta_I
(G))\stackrel{(j_{I*}\circ f_I^{*})}
{\rightarrow}
 {\rm CH}(\mathcal{S}_{d,univ}^{N/B})$$
is surjective, where $f_I$ is the restriction of
$f_N$ to $f_N^{-1}(\Delta_I(G))\subset G^N$ and $j_I$ is the inclusion
of $f_N^{-1}(\Delta_I(G))$ in $\mathcal{S}_d^{N/U}$. Note that
$f_I^{-1}(\Delta_I(G))=\Delta_{I/B}(\mathcal{S}_{d,univ})$.
 Finally, we observe that
the restriction map
$${\rm CH}(G^N)\rightarrow {\rm CH}(\Delta_I(G))$$
is surjective, and that for any $\alpha\in {\rm CH}(G^N)$,
$$j_{I*}
\circ f_I^*(\alpha_{\mid \Delta_I(G)})=f_N^*\alpha\cdot (j_{I*}
\circ f_I^*)(1)=f_N^*\alpha\cdot
\Delta_{I/B}(\mathcal{S}_{d,univ}),$$
and this finishes the proof.

\end{proof}
\begin{coro} \label{corpourthefinsd} For any universally defined cycle $z$ on $N$-th powers of surfaces and for sufficiently large $d$,
there exists a polynomial $P_d$ with rational coefficients, depending only on $z$ and $d$ such that
for any smooth complete intersection surface
$S_d\subset G$ as above,
\begin{eqnarray}
\label{eqpourPd}
z_{S_d}=P_d(pr_i^*c_1(S_d),pr_j^*c_2(S_d),pr_{rs}^*\Delta_{S_d})\,\,{\rm in}\,\,{\rm CH}(S_d^N)_\mathbb{Q}.
\end{eqnarray}
Furthermore, $(4d-5)^{2N}P_d$ has integral coefficients.
\end{coro}
\begin{proof} As $z$ is universal, there exists a cycle $z_{\mathcal{S}_d}\in {\rm CH}(\mathcal{S}_d^{N/B})$ such that for any surface $S_d$ as above,
$$z_{S_d}=(z_{\mathcal{S}_d})_{\mid S_d^N},$$
where we see $S_d$ as a fiber of the universal family
$\mathcal{S}_d\rightarrow B$.
We next use Proposition \ref{leunivsurfd} to write, for $d>>0$, $\mathcal{Z}_d$ as a combination
$\sum_I f_N^*\alpha_I\cdot \Delta_{I/B}(\mathcal{S}_d)$, where $\alpha_I\in {\rm CH}(G(2,5))$.
Furthermore, is it immediate to prove that ${\rm CH}(G^N)={\rm CH}(G)^{\otimes N}$, so that
we can write each $\alpha_I$ as a polynomial with integral coefficients in
$pr_i^*m,\,m:=c_1(\mathcal{O}_G(1))=c_1(E)$ and $pr_j^*c,\,c:=c_2(E)$ where $E$ is the
dual of the tautological subbundle on $G$. Of course, under restriction to $S_d^N$, only polynomials of weighted degree $\leq 2$ in each set of
 variables $pr_i^*m,\,pr_i^*c$ will survive.
We now observe that the restriction of $m$ to $S_d$ is a rational multiple of $c_1(S)$ (more precisely,
$K_{S_d}=
\mathcal{O}_G(4d-5)_{\mid S}$ by the adjunction formula), and the restriction of $c$ to $S_d$ is an adequate
linear combination  of $\frac{1}{(4d-5)^2}c_1(S_d)^2,\,c_2(S_d)$.
Putting everything together and using the fact
that the relative diagonals $\Delta_{I/B}(\mathcal{S}_d)$ restrict to
$\Delta_I(S_d)$, we get the result.

\end{proof}
\begin{rema}\label{remaaddendum}{\rm Note that Corollary \ref{corpourthefinsd} is true
more generally
for the regular and complete intersection
 locus $S_{reg}$ of any set of $4$ degree $d$ equations on $G$.
 The proof  uses Proposition \ref{leunivsurfd} (which works for  the  family
 $\mathcal{S}_d\rightarrow B$ of smooth complete intersection quasi-projective surfaces),  and
  both conditions (i) and  (ii)  in Definition \ref{defiunivdef}.}
\end{rema}
The corollary above proves Theorem \ref{theopourtheofin} for smooth complete intersection surfaces
of degree $d$, and more generally
for the regular and complete intersection
 locus of any set of $4$ degree $d$ equations on $G$. What remains to be done is to prove that the polynomial above works for all
surfaces.
Note that the polynomial $P_d$ is in fact not uniquely defined
as only its value on the set of variables
$pr_i^*c_1(S_d),\,pr_j^*c_2(S_d),\,pr_{st}^*\Delta_{S_d}$ is well defined in
${\rm CH}(S_d^N)_\mathbb{Q}$. Hence {\it a priori} $P_d$ is only defined modulo the relations
in ${\rm CH}(S_d^N)_\mathbb{Q}$
between these variables. However, the following result shows that
a part of $P_d$ is in fact independent of $d$ for large $d$.

\begin{prop}\label{coroonnesaitjamais} For any universally
defined cycle $z$  on $N$-th  powers of surfaces, there exists a polynomial $Q$
in the variables $pr_{st}^*\Delta_S$, depending only on $z$,  with the following property:
For any smooth surface $S$, there is a Zariski dense open set
$V\subset S$ such that
$z_V=Q(pr_{st}^*\Delta_S)$ in ${\rm CH}(V^N)_\mathbb{Q}$.

\end{prop}

\begin{proof} Let $Q_d$ be the part of the polynomial $P_d$ which involves only the diagonals.
Then let $U_d\subset S_d$ be the complement of a hyperplane section
defined by the choice of a codimension $2$ vector subspace $W\subset \mathbb{C}^5$ in general position. As
$c_1(\mathcal{O}_{S_d}(1))$ and $c_2(E)$ vanish in ${\rm CH}(U_d)$, we deduce
from (\ref{eqpourPd}) that
\begin{eqnarray}\label{eqncyrestUd}
z_{U_d}=Q_d(pr_{st}^*\Delta_{U_d})\,\,{\rm in}\,\,{\rm CH}(U_d^N).
\end{eqnarray}
We observe now that
for $d'\leq d$, a surface $S_{d'}$ which is the complete intersection in $G$
of hypersurfaces of degree $d'$ is an irreducible component
of a (singular) surface $\Sigma_d=S_{d'}\cup T$ defined as the complete
intersection in $G$ of four degree $d$ hypersurfaces
containing $S_{d'}$ and that, denoting
 $C:=S_{d'}\cap T$, the open set  $U'_{d'}:=S_{d'}\setminus C$ is contained in the smooth locus
 of $\Sigma_d$. From Remark \ref{remaaddendum},
 we thus get that
 \begin{eqnarray}\label{eqncyrestUdsurU'}
z_{U'_{d'}}=P_d(pr_i^*c_1(U'_{d'}),pr_j^*c_2(U'_{d'}),pr_{st}^*\Delta_{U'_{d'}})\,\,{\rm in}\,\,{\rm CH}({U'_{d'}}^N)_\mathbb{Q},
\end{eqnarray}
and after restriction
to $V_{d'}:=U_{d'}\cap U'_{d'}$, this becomes
\begin{eqnarray}\label{eqncyrestUdsurU'surV}
z_{V_{d'}}=Q_d(pr_{st}^*\Delta_{V_{d'}})\,\,{\rm in}\,\,{\rm CH}({V}_{d'}^N)_\mathbb{Q},
\end{eqnarray}
On the other hand, we also have (\ref{eqncyrestUd}) for $d'$, which provides after restriction
to $V_{d'}$

\begin{eqnarray}\label{eqncyrestUdsurd'}
z_{V_{d'}}=Q_{d'}(pr_{st}^*\Delta_{V_{d'}})\,\,{\rm in}\,\,{\rm CH}({V}_{d'}^N)_\mathbb{Q}.
\end{eqnarray}
Hence $Q_d-Q_{d'}$ belongs to the kernel
of the map
$${\rm ev}_{d'}:\mathbb{Q}[X_{rs}]_{1\leq r\not=s\leq N}\rightarrow {\rm CH}(V^N),$$
$$f\mapsto f(p_{rs}^*\Delta_V),$$
where $V$ is a sufficiently small Zariski  open set of a general
complete intersection of four hypersurfaces of degree $d'$ in $G$.
On the other hand, it follows from the above construction
that ${\rm Ker}\,{\rm ev}_d\subset {\rm Ker}\,{\rm ev}_{d'}$ for $d'\leq d$. As the polynomials we consider are homogeneous of given degree (equal to half the codimension of $z$),
they live in a finite dimensional vector space and
we conclude that these kernels are in fact stationary, equal to $K$ for $d\geq d_0$. So we
finally conclude that there exists a $d_0$ such that
for $d\geq d'\geq d_0$,
$Q_d-Q_{d'}$ belongs to $K$. It follows that for any $d$, for any reduced
complete intersection
of four degree $d$ hypersurfaces in $G$, and for a dense Zariski open set $V\subset S$
$$z_{V}=Q_{d_0}(pr_{st}^*\Delta_{V}
)\,\,{\rm in}\,\,{\rm CH}(V^N)_\mathbb{Q}.$$
As any smooth quasi-projective surface has a dense  Zariski open set which
is contained in the smooth locus of such a complete intersection for $d$ large enough,
the proposition is proved, with $Q=Q_{d_0}$.

\end{proof}
We finish this section with the proof of  Theorem \ref{theopourtheofin} in the case $N=1$.
\begin{prop}\label{projusteS} Let $z$ be
a universally defined cycle on surfaces.
Then there is a polynomial $P$  independent of $S$ and with integral coefficients, such that
for any smooth quasi-projective surface $S$,
$$z_S=P(c_2(S),\,c_1(S))\in {\rm CH}(S)_\mathbb{Q}.$$

\end{prop}
\begin{proof}
Let us first treat the case of $z\in {\rm CH}^1(S)$ universally defined.
For complete intersections $S_d$  of four hypersurfaces of degree
$d$  in $G$, we must have by Corollary \ref{corpourthefinsd}
$$z=\alpha_d K_{S_d},$$
for some rational number $\alpha_d$ such that $(4d-5)\alpha_d\in\mathbb{Z}$,
and for any surface $S$, choosing a very  ample vector bundle
$E$ of rank $2$ on $S$ to embed $S$ in $G$, and choosing $d$ large enough, we get
\begin{eqnarray}
\label{eqfin4nov}z_{S\mid U}=\alpha_dK_U\,\,{\rm in}\,\,{\rm CH}^1(U)_\mathbb{Q},
\end{eqnarray}
where $U=S\setminus C$, the surface
$S\cup_CT=\Sigma_d$ being the complete intersection of four degree $d$ hypersurfaces
containing $S$ in $G$.
The curve $C$ belongs to the linear system
$|(4d-5)L-K_S|$, where $L={\rm det}\,E=\mathcal{O}_G(1)_{\mid S}$. For a general choice of
equations and $d$ large enough,
the curve $C$ will be irreducible, so by the localization exact sequence,
(\ref{eqfin4nov}) rewrites as
\begin{eqnarray}
\label{eqfin4nov}z_{S}=\alpha_dK_S+\beta_d C\,\,{\rm in}\,\,{\rm CH}^1(S)_\mathbb{Q},
\end{eqnarray}
If $K_S$ and $L$ are linearly independent in ${\rm CH}^1(S)_\mathbb{Q}$, this implies, because
the left hand side is independent of $L$,  that $\beta_d=0$ and thus
$z_S=\alpha_d K_S$, with $\alpha_d=:\alpha$ necessarily independent of $d$.
If not, we simply blow up $S$ at one point and choose $L$ on
$\widetilde{S}$ linearly independent of $K_{\widetilde{S}}$ in ${\rm CH}^1(\widetilde{S})_\mathbb{Q}$.
Then the above conclusion applies to $\widetilde{S}$, hence we get
\begin{eqnarray}
\label{eqfin4novtructruc}z_{\widetilde{S}}=\alpha K_{\widetilde{S}}\,\,{\rm in}\,\,{\rm CH}^1(\widetilde{S}).
\end{eqnarray}
 As
$$S\setminus\{p\}\cong \widetilde{S}\setminus E_p,\,\,{\rm CH}^1(S)_\mathbb{Q}\cong
{\rm CH}^1(S\setminus\{p\})_\mathbb{Q},$$
(\ref{eqfin4novtructruc})  is also true for $S$ by condition (ii) in Definition
\ref{defiunivdef}. Finally $\alpha$ has to be an integer since $(4d-5)\alpha\in \mathbb{Z}$
for any $d$.
This proves the result for $z_S\in {\rm CH}^1(S)$ universally defined.

Let us now prove the result for a  universally defined cycle $z_S\in {\rm CH}^2(S)$.
We start the proof exactly as before, and we have by Corollary \ref{corpourthefinsd} and Remark
\ref{remaaddendum} that
for $S$ the regular locus of  a complete intersection of four hypersurfaces of degree $d$ in $G$,
$z_S=\mu_d c_2(S)+\nu_d c_1(S)^2$ in ${\rm CH}_0(S)_\mathbb{Q}$.
With the notation above, we conclude as before that for any pair $(S,E) $ consisting of a
smooth surface $S$ and a very ample rank $2$ vector bundle in a bounded family
(depending on $d$), we have
\begin{eqnarray}
\label{equtilepourtheofin66}
z_{S\mid U}=\mu_d c_2(S)_{\mid U}+\nu_dc_1(S)^2_{\mid U}\,\,{\rm in}\,\,{\rm CH}_0(U)_\mathbb{Q},
\end{eqnarray}
where $U=S\setminus C$ and the curve $C$ is defined as above.
By the localization exact sequence, this is equivalent
to
\begin{eqnarray}
\label{equtilepourtheofin77}
z_{S}=\mu_d c_2(S)+\nu_dc_1(S)^2+z'_{S}\,\,{\rm in}\,\,{\rm CH}_0(S)_\mathbb{Q},
\end{eqnarray}
where $z'_S$ is a $0$-cycle supported on $C$.
 If  $d$ is large enough (with respect to the given
bounded family), the surface $S$ is schematically
 cut out by hypersurfaces of degree $d$ in $G$ and this implies
  that the morphism of sheaves
$$I_S(d)\otimes\mathcal{O}_S\rightarrow N_{S/G}^*(d),\,\,N_{S/G}^*=\mathcal{I}_S/\mathcal{I}_S^2,$$
is surjective. The curve $C$ is the curve obtained as the degeneracy locus of
the morphism
$$f:\mathcal{O}_S^4\rightarrow N_{S/G}^*(d),$$
deduced from the choice of $4$ degree $d$ equations defining $S$ in $G$.
It is in fact better to see such a curve as
embedded in
$\mathbb{P}(N_{S/G}(-d))$, as the complete intersection of $4$ hypersurfaces
in the linear system $|\mathcal{O}_{\mathbb{P}(N_{S/G}(-d))
}(1)|$.
Let $B_S$ be the Zariski open set of $I_S(d)^4$ consisting of the
$(f_1,\ldots, f_4)\in I_S(d)^4$ such that
the $f_i$'s form  a regular sequence  cutting $S$ at its generic point.
For each element $b\in B_S$, let $C_b\subset \mathbb{P}(N_{S/G}(-d))$ be the corresponding curve.
We have
a family  of complete intersection surfaces
$$\mathcal{S}'\rightarrow B_S $$
which is a subfamily of the family
$\mathcal{S}'_d\rightarrow B$ of all complete intersection surfaces of four degree $d$
hypersurfaces in $G$.
The family
$\mathcal{S}'\rightarrow B_S $ contains $B_S\times S$, which itself
contains the family of curves
$\mathcal{C}\rightarrow B_S$ with fiber $C_b$ over $b\in B_S$. The surfaces
$b\times (S\setminus C_b)$ are contained the smooth locus of
the surfaces $S'_b$.
In conclusion, the inclusion
$$\mathcal{S}'\subset \mathcal{S}'_d$$
of families of complete intersection surfaces restricts, by considering the Zariski
open sets where the surfaces are smooth, to
 an open inclusion of smooth families of surfaces
$$(B_S\times S)\setminus \mathcal{C}\hookrightarrow \mathcal{S}_d\times_BB_S.$$
Applying the conditions (i) and (ii)
in Definition \ref{defiunivdef} and Proposition \ref{leunivsurfd}, we get
that
$$(pr_S^*(\mu_d c_2(S)+\nu_d c_1(S)^2))_{\mid (B_S\times S)\setminus \mathcal{C}}=(z_{B_S\times S})_{\mid (B_S\times S)\setminus \mathcal{C}}\,\,{\rm in}\,\,{\rm CH}((B_S\times S)\setminus \mathcal{C})_\mathbb{Q},$$
which rewrites equivalently, by the localization exact sequence,
as
\begin{eqnarray} \label{equtiledu24oct} (pr_S^*(\mu_d c_2(S)+\nu_d c_1(S)^2))-z_{B_S\times S}=i_*(w)\,\,{\rm in}\,\,{\rm CH}(B_S\times S)_\mathbb{Q},
\end{eqnarray}
where $w\in {\rm CH}^1(\mathcal{C})_\mathbb{Q}$ and $i:\mathcal{C}\rightarrow B_S\times S$ is the natural morphism. Here, as already mentioned, it is much better to consider the variant of
$\mathcal{C}$ which is contained in $B_S\times \mathbb{P}(N_{S/G}(-d))$, (and the
morphism $i$ is only at the general point of
$B_S$ an embedding) because it is then the universal complete intersection
of four hypersurfaces and this allows to conclude with exactly the same proof as in Proposition
\ref{leunivsurfd}:
\begin{lemm}\label{propCHCuniv} The restriction map ${\rm CH}^1(\mathbb{P}(N_{S/G}(-d)))
\rightarrow {\rm CH}^1(\mathcal{C})$ is surjective.
\end{lemm}

Note that $${\rm CH}^1(\mathbb{P}(N_{S/G}(-d)))={\rm CH}^1(S)\oplus \mathbb{Z}c_1(\mathcal{O}_{\mathbb{P}(N_{S/G}(-d))}(1)).$$
The cycle $w$ of (\ref{equtiledu24oct}) thus decomposes as
$w=(w_S)_{\mid C}+(w_P)_{\mid C}$, with
$$w_P=\alpha c_1(\mathcal{O}_{\mathbb{P}(N_{S/G}(-d))}(1))\in {\rm CH}^1(\mathbb{P}(N_{S/G}(-d)))_\mathbb{Q}.$$
 As
 $C_b\subset \mathbb{P}(N_{S/G}(-d))$ is the intersection of four
 members  of the linear system $|\mathcal{O}_{\mathbb{P}(N_{S/G}(-d))}(1)) |$ for any given $b\in B_S$, we get
\begin{eqnarray}
\label{eqnou24oct}
i_*((w_P)_{\mid C_b})=\alpha s_2(N_{S/G}(-d)),
\end{eqnarray}
where $s_2$ is the second Segre class.
Note  that by (\ref{eqs2}), $s_2(N_{S/G}(-d))
$ can be explicitly computed as
a linear combination with integral coefficients of  $L^2,\,c_2(E), c_2(S),\,c_1(S)^2$, involving non trivially
$c_2(E)$.
Thus we get from (\ref{equtiledu24oct}) and (\ref{eqnou24oct}), by restricting to any point
$b\in B_S$
\begin{eqnarray}\label{eqder5no} z_S=\mu_d c_2(S)+\nu_d c_1(S)^2)-\alpha s_2(N_{S/G}(-d))-((4d-5)L-K_S)\cdot w_S\,\, {\rm in}\,\,{\rm CH}_0(S)_\mathbb{Q}.
\end{eqnarray}

Let us  analyze in general the cycle $w_S\in {\rm CH}^1(S)$.
We observe first that this cycle is universally defined for triples
$(S,E,d)$, where $d$ has to be large and we conclude as in the previous
proof that there are rational numbers $\gamma,\,\delta$
such that
$$w_S=\gamma c_1(S)+\delta L,$$
where $L=c_1(E)$.
Combining this with
(\ref{eqder5no}), we get, using the fact that the
curve $i(C_b)\subset S$ belongs to the linear system  $|(4d-5)L-K_S|$:

\begin{eqnarray} \label{equtiledu24oct11} z_{S}=\mu_d c_2(S)+\nu_dc_1(S)^2
-\alpha s_2(N_{S/G}(-d))
\\
\nonumber
-((4d-5)L-K_S)\cdot (\gamma K_S+\delta L)\,\,{\rm in}\,\,{\rm CH}_0( S)_\mathbb{Q},
\end{eqnarray}
where the coefficients $\alpha,\,\beta,\,\gamma$ a priori depend on $S,\,E$ and $d$.
Finally we observe that the left hand side depends only on $S$, while in the right hand side,
for fixed large $d$, we can freely change $c_2(E),\,L=c_1(E)$ (staying in a bounded family).
Assume first that no nonzero linear combination of the cycles  $s_2(N_{S/G}(-d)),\,L^2, \, K_S\cdot L$  belongs to
the linear span of the cycles
$z_{S},\,c_2(S),\,c_1(S)^2$. Then (\ref{equtiledu24oct11}) easily implies that
$\alpha=\delta=\gamma=0$, so that we get in this case
\begin{eqnarray} \label{equtiledu24oct222} z_{S}=\mu_d c_2(S)+\nu_dc_1(S)^2
\,\,{\rm in}\,\,{\rm CH}_0( S)_\mathbb{Q}.
\end{eqnarray}
To treat the general case, we will first prove the result assuming
${\rm CH}_0(S)$  is nontrivial, that is, if $S$ is projective
and connected, ${\rm CH}_0(S)\not=\mathbb{Z}$. It then easily follows that ${\rm CH}_0(S)$ is uncountable,
and one deduces immediately that for the blow-up $S'$ of $S$ at $6$ general points, and for adequate
choice of $E'$, (and thus $L'={\rm det}\,E'$)
no combination of the cycles  $s_2(N_{S'/G}(-d)),\,L'^2, \, K_{S'}\cdot L'$  belongs to
the linear span of the cycles
$z_{S'},\,c_2(S'),\,c_1(S')^2$. Thus we conclude from (\ref{equtiledu24oct222}) that
for a general blow-up $S'$ of $S$ along $6$ points,
\begin{eqnarray} \label{equtiledu24oct333}z_{S'}=\mu_d c_2(S')+\nu_dc_1(S')^2
\,\,{\rm in}\,\,{\rm CH}_0( S')_\mathbb{Q}.
\end{eqnarray}
But this immediately implies
that for the original surface $S$, (\ref{equtiledu24oct222}) also holds, because
we can restrict (\ref{equtiledu24oct333}) to
the Zariski open set $S'\setminus \cup_{i=1}^6E_i=S\setminus \{p_1,\ldots,p_6\}$,
where the $p_i$'s are the points blown-up in $S$ and the $E_i$'s are
the exceptional curves over them.
Using property (ii) in Definition \ref{defiunivdef}, we find that
(\ref{equtiledu24oct333}) provides:
\begin{eqnarray} \label{equtiledu24oct333isb}(z_{S})_{\mid S\setminus \{p_1,\ldots,p_6\}}=\mu_d c_2(S)+\nu_dc_1(S)^2
\,\,{\rm in}\,\,{\rm CH}_0( S\setminus \{p_1,\ldots,p_6\})_\mathbb{Q}.
\end{eqnarray}
As the $p_i$'s are general points, this immediately implies that
$$z_{S}=\mu_d c_2(S)+\nu_dc_1(S)^2
\,\,{\rm in}\,\,{\rm CH}_0( S)_\mathbb{Q}.$$
Finally, the above discussion was concerning polarized surfaces in a bounded family
(depending on $d$). Taking $d$ larger, we conclude that the coefficients
$\mu_d$ and $\nu_d$ do not depend on $d$, which concludes the proof, under the assumption made
on ${\rm CH}_0(S)$.

It now only remains to prove the result for surfaces with ${\rm CH}_0(S)=\mathbb{Z}$.
In this case, we can use the following trick, using the fact that we already know that
$$\mu_d=\mu,\,\,\nu_d=\nu$$
are independent of $d$ (supposed to be large), so that
 $$z_{S}=\mu c_2(S)+\nu c_1(S)^2$$
 when $S$ is the regular locus of the complete intersection
of four hypersurfaces of degree $d$ in $G$. Here $\mu,\,\nu$ are {\it a priori} rational numbers
but in fact, they are integers because $(4d-5)^5\mu_d,\,(4d-5)^2\nu_d$ are
integers by Corollary \ref{corpourthefinsd}.
We observe now that with the  notation introduced for  Proposition
\ref{leunivsurfd},
the equality
$$z_{\mathcal{S}_d}=\mu c_2(K_{\mathcal{S}_d/B})+\nu c_1(K_{\mathcal{S}_d/B})^2$$
must hold in ${\rm CH}^2(\mathcal{S}_d)$, because both sides are integral cycles,
the equality is satisfied in ${\rm CH}^2(\mathcal{S}_d)_\mathbb{Q}$, and
 ${\rm CH}^2(\mathcal{S}_d)$ has no torsion.
 We choose as before a very ample rank $2$ vector bundle $E$ on $S$ with $c_2(E)=c\in {\rm CH}_0(S)=\mathbb{Z}$, and embed
$S$ in $G$ using $5$ general sections of $E$. For large $d$, we then
have as before
the family $(B_S\times S)\setminus \mathcal{C}\subset \mathcal{S}_d$, where
$B_S$ is a Zariski open set of $I_S(d)^4$, and
we conclude by restriction, using (i) and (ii) of Definition
\ref{defiunivdef} that
$$(pr_S^*z_S)_{\mid (B_S\times S)\setminus \mathcal{C}}=(\mu pr_S^*c_2(K_{S})+\nu pr_S^*c_1(K_S)^2)_{\mid (B_S\times S)\setminus \mathcal{C}}$$
holds in  ${\rm CH}^2((B_S\times S)\setminus \mathcal{C})$.
In other words, we proved the same equalities as before, but in ${\rm CH}$ instead of
${\rm CH}_\mathbb{Q}$.
We now apply the localization exact sequence and Lemma \ref{propCHCuniv} to conclude
that (\ref{eqder5no})
holds in fact with integral cycles, namely
\begin{eqnarray}\label{eqder5nopasder} z_S=\mu c_2(S)+\nu c_1(S)^2-\alpha s_2(N_{S/G}(-d))-((4d-5)L-K_S)\cdot w_S\,\, {\rm in}\,\,{\rm CH}_0(S),
\end{eqnarray}
where $\alpha$ is an integer and $w_S\in {\rm CH}^1(S)$.

We now choose $d$, $c$ and $L={\rm det}\,E$
in such a way that an arbitrarily large given integer $M$ divides
  both $s_2(N_{S/G}(-d))$ and $(4d-5)L-K_S\in {\rm NS}(S)$. This is possible because
  \begin{eqnarray}
  \label{eqs2}s_2(N_{S/G}(-d))=c+Q(L,K_S),
  \end{eqnarray}
  where $Q$ is a degree $2$ polynomial with integer coefficients.
  Then  we apply
formula (\ref{eqder5nopasder}), and we
 conclude that
$M$ divides $z_S-\mu c_2(S)-\nu c_1(S)^2 $ in ${\rm CH}_0(S)=\mathbb{Z}$. As
$M$ is arbitrarily large, it follows that
$z_S=\mu c_2(S)+\nu c_1(S)^2$.
\end{proof}

 Institut de Math\'ematiques de Jussieu
 
 4 place Jussieu

Case 247

75252 Paris Cedex 05

 France

\smallskip
 claire.voisin@imj-prg.fr
    \end{document}